\documentclass[11pt]{amsart} 
\usepackage[margin=3cm]{geometry}

\usepackage{verbatim}
\usepackage{amssymb}
\usepackage{amsbsy}
\usepackage{amscd}
\usepackage{amsmath}
\usepackage{amsthm}
\usepackage{amsxtra}
\usepackage[mathscr]{eucal}

\usepackage{xypic}
\usepackage{epic,eepic} 
\usepackage{hyperref}
\usepackage[all]{hypcap}
\hypersetup{colorlinks,bookmarksopen,bookmarksnumbered,citecolor=blue,pdfstartview=FitH}
\usepackage{xcolor}
\hypersetup{
    colorlinks,
    linkcolor={red!50!black},
    citecolor={blue!50!black},
    urlcolor={blue!80!black}
}
\usepackage[OT2,OT1]{fontenc}
\newcommand\cyr
{
\renewcommand\rmdefault{wncyr}
\renewcommand\sfdefault{wncyss}
\renewcommand\encodingdefault{OT2}
\normalfont
\selectfont
}
\DeclareTextFontCommand{\textcyr}{\cyr}

\newtheorem{theorem}{Theorem}[section]
\newtheorem*{theorem*}{Theorem}
\newtheorem{proposition}[theorem]{Proposition}
\newtheorem*{proposition*}{Proposition}
\newtheorem{lemma}[theorem]{Lemma}
\newtheorem{corollary}[theorem]{Corollary}
\newtheorem*{corollary*}{Corollary}

\newtheorem{conjecture}[theorem]{Conjecture} 

\theoremstyle{definition}
\newtheorem{example}[theorem]{Example}
\newtheorem{definition}[theorem]{Definition}
\newtheorem{condition}[theorem]{Condition} 
\newtheorem{problem}[theorem]{Problem} 

\theoremstyle{remark}
\newtheorem{remark}[theorem]{Remark}

\def \AA {\mathcal{A}}

\def \E {\mathcal{E}}
\def \G {\mathbb{G}}

\def \cR {\mathcal{R}}
\def \L {\mathcal{L}}
\def \N {\mathcal{N}}
\def \O {\mathcal{O}}
  
\def \T {\mathcal{T}}

\def \V {\mathcal{V}}

\def \C {\mathbb{C}}
\def \F {\mathbb{F}}
\def \P {\mathbb{P}}
\def \Q {\mathbb{Q}}
\def \R {\mathbb{R}}
\def \Z {\mathbb{Z}}

\def \iu {\sqrt{-1}}

\def \Gg {\widehat{\Gamma}} 
\DeclareMathOperator{\ch}{ch} 
\DeclareMathOperator{\Ch}{Ch} 

\def \im {\mathrm{Im}}

\def \ge {\geqslant}

\def \le {\leqslant}

\newcommand{\ang}[1] {\langle #1 \rangle}
\newcommand{\Ang}[1]{\left\langle #1 \right\rangle} 
\newcommand{\ov}[1] {\overline{#1}}

\newcommand{\Ker}{\operatorname{Ker}}

\newcommand{\udot}{{\:\raisebox{3pt}{\text{\circle*{1.5}}}}}
\newcommand{\ldot}{{\:\raisebox{2pt}{\text{\circle*{1.5}}\:\!}}}

\DeclareMathOperator{\Hom}{Hom}

\DeclareMathOperator{\Ext}{Ext}

\DeclareMathOperator{\rank}{rank}

\DeclareMathOperator{\End}{End} 
\DeclareMathOperator{\id}{id} 
\DeclareMathOperator{\Gr}{Gr}

\def\parfrac#1#2{{\frac{\partial #1}{\partial #2}}}

\newcommand{\Eff}{\operatorname{Eff}}

\newcommand{\re}{\operatorname{Re}} 
 
\newcommand{\Td}{\operatorname{Td}} 
\newcommand{\td}{\operatorname{td}} 
\newcommand{\pt}{\operatorname{pt}}

\newcommand{\Euler}{\operatorname{Euler}} 
\newcommand{\sgn}{\operatorname{sgn}}

\newcommand{\vol}{\operatorname{vol}}
\newcommand{\Lie}{\operatorname{Lie}} 
\newcommand{\Const}{\operatorname{Const}} 
\newcommand{\PD}{\operatorname{PD}}

\newcommand{\tJ}{\widetilde{J}}

\newcommand{\halpha}{\widehat{\alpha}} 
 
\newcommand{\mY}{Y^-}

\newcommand{\tg}{\tilde{g}}
\newcommand{\tv}{\tilde{v}}

\newcommand{\hG}{\widehat{G}} 

\newcommand{\frS}{\mathfrak{S}}

\newcommand{\cF}{\mathcal{F}}
\newcommand{\cG}{\mathcal{G}}

\title[Gamma conjecture via mirror symmetry]
{Gamma conjecture via mirror symmetry}

\author[S. Galkin]{Sergey Galkin}

\address{National Research University Higher School of Economics, Russian Federation}

\email{Sergey.Galkin@phystech.edu} 

\author[H. Iritani]{Hiroshi Iritani} 

\address{Department of Mathematics \\ 
Graduate School of Science \\ 
Kyoto University \\  
Kitashirakawa-Oiwake-cho \\ Sakyo-ku \\ 
Kyoto \\ 606-8502 \\ Japan} 

\email{iritani@math.kyoto-u.ac.jp} 



\begin{document}

\begin{abstract} 
The asymptotic behaviour of solutions to the quantum 
differential equation of a Fano manifold $F$ 
defines a characteristic class $A_F$ of $F$, called 
the principal asymptotic class. 
Gamma conjecture \cite{GGI:gammagrass} of Vasily Golyshev 
and the present authors 
claims that the principal asymptotic class $A_F$ 
equals the Gamma class $\Gg_F$ associated to Euler's 
$\Gamma$-function. 
We illustrate in the case of toric varieties, toric complete intersections 
and Grassmannians how this conjecture follows from 
mirror symmetry. We also prove that Gamma conjecture is 
compatible with taking hyperplane sections, and 
give a heuristic argument how the mirror oscillatory integral 
and the Gamma class for the projective space arise from the 
polynomial loop space. 
\end{abstract}

\maketitle
  
\let\oldtocsection=\tocsection
\let\oldtocsubsection=\tocsubsection

\renewcommand{\tocsection}[2]{\hspace{0em}\oldtocsection{#1}{#2}}
\renewcommand{\tocsubsection}[2]{\hspace{1em}\oldtocsubsection{#1}{#2}}
{
\tableofcontents
} 

\section{Introduction} 
\subsection{Gamma conjecture} 
Gamma conjecture \cite{GGI:gammagrass} is a conjecture 
which relates quantum cohomology of a Fano manifold 
with its topology. 
The small quantum cohomology of $F$ defines a flat connection 
(quantum connection) 
over $\C^\times$ and its solution is given by 
a (multi-valued) cohomology-valued function $J_F(t)$ called the \emph{$J$-function}. 
Under a certain condition (Property $\O$), the limit of the $J$-function: 
\[
A_F := \lim_{t\to +\infty}  \frac{J_F(t)}{\Ang{[\pt],J_F(t)}} \in  H^\udot(F) 
\]
exists and defines the \emph{principal asymptotic class} $A_F$ of $F$. 
\emph{Gamma conjecture I} says that $A_F$ equals the Gamma class 
$\Gg_F=\Gg(TF)$ of the tangent bundle of $F$ (see \S \ref{subsec:Gamma_class}): 
\[
A_F = \Gg_F.  
\]
Suppose that the quantum cohomology of $F$ is semisimple. 
In this case, we can define \emph{higher} asymptotic classes $A_{F,i}$, 
$1\le i \le N=\dim H^\udot(F)$ from exponential 
asymptotics of flat sections of the quantum connection. 
\emph{Gamma conjecture II} says 
that there exists a full exceptional collection $E_1, E_2,\dots,E_N$ 
of $D^b_{\rm coh}(F)$ such that we have 
\[
A_{F,i} = \Gg_F \cdot \Ch(E_i) \qquad 
i=1,\dots,N. 
\]
Here we write $\Ch(E) := (2\pi\iu)^{\frac{\deg}{2}} \ch(E) 
= \sum_{p=0}^{\dim F} (2\pi\iu)^p \ch_p(E)$ for the 
$(2\pi\iu)$-modified version of the Chern character.  
The principal asymptotic class $A_F$ corresponds to the 
exceptional object $E = \O_F$. 
\subsection{Riemann--Roch} 
\label{subsec:HRR} 
One may view Gamma conjecture as a \emph{square root}  
of the index theorem. Recall the Hirzebruch--Riemann--Roch 
formula: 
\[
\chi(E_1,E_2) = \int_F \ch(E_1^\vee) \cdot \ch(E_2) \cdot \td_F 
\]
for vector bundles $E_1,E_2$ on $F$, where $\chi(E_1,E_2) 
= \sum_{i=0}^{\dim F} (-1)^i\dim \Ext^i(E_1,E_2)$ 
is the Euler pairing and $\td_F =\td(TF)$ is the Todd class 
of $F$. The famous identity 
\[
\frac{x}{1-e^{-x}} = e^{x/2} \Gamma\left(1-\frac{x}{2\pi\iu}\right) 
\Gamma\left(1+ \frac{x}{2\pi\iu}\right) 
\]
gives the factorization 
$(2\pi\iu)^{\frac{\deg}{2}}
\td_F =e^{\pi \iu c_1(F)} \cdot \Gg_F \cdot \Gg_F^*$  
of the Todd class\footnote
{We define $(2\pi\iu)^{\frac{\deg}{2}} \td_F := 
\sum_{p=0}^{\dim F} (2\pi \iu)^p t_p$ if $\td_F = \sum_{p=0}^{\dim F} 
t_p$ with $t_p \in H^{2p}(F)$.}, which in turn factorizes 
the Hirzebruch--Riemann--Roch formula as:
\begin{equation} 
\label{eq:factorize_HRR} 
\chi(E_1,E_2) = \left[\Ch(E_1) \cdot \Gg_F, \Ch(E_2) \cdot \Gg_F\right). 
\end{equation} 
Here we set 
$\Gg_F^* = (-1)^{\frac{\deg}{2}} \Gg_F = \sum_{p=0}^{\dim F} 
(-1)^p \gamma_p$ (if we write $\Gg_F = \sum_{p=0}^{\dim F} \gamma_p$ 
with $\gamma_p \in H^{2p}(F)$)  
and 
$[\cdot,\cdot)$ is a non-symmetric pairing on $H^\udot(F)$ 
given by 
\begin{equation} 
\label{eq:pairing_[)}
[\alpha,\beta) = \frac{1}{(2\pi)^{\dim F}} 
\int_F (e^{\pi\iu c_1(X)} e^{\pi \iu \mu} \alpha ) \cup \beta
\end{equation}
with $\mu\in \End(H^\udot(F))$ defined by
$\mu(\phi) = (p- \frac{\dim F}{2}) \phi$ 
for $\phi\in H^{2p}(F)$. 
Via the factorization \eqref{eq:factorize_HRR}, 
Gamma conjecture II implies part of Dubrovin's 
conjecture \cite{Dubrovin98}: the Euler matrix 
$\chi(E_i,E_j)$ of the exceptional collection equals the 
Stokes matrix $S_{ij} = [A_{F,i}, A_{F,j})$ of 
the quantum differential equation. 

\subsection{Mirror symmetry} 
In the B-side of mirror symmetry, solutions to Picard--Fuchs equations 
are often given by hypergeometric series whose coefficients 
are ratios of $\Gamma$-functions. 
Recall that a mirror of a quintic threefold $Q \subset \P^4$ 
is given by the pencil of hypersurfaces 
$Y_t = \{f(x) =t^{-1}\}$ 
in the torus $(\C^\times)^4$ which can be compactified to 
smooth Calabi--Yau threefolds $\ov{Y}_t$ \cite{CDGP}, 
where $f$ is the Laurent polynomial given by 
\[
f(x) = x_1 + x_2 + x_3 + x_4 + \frac{1}{x_1 x_2 x_3 x_4}. 
\]
For the holomorphic volume form 
$\Omega_t = d \log x_1 \wedge \cdots \wedge d\log x_4 /df$ 
on $\ov{Y}_t$ and a real 3-cycle $C \subset \ov{Y}_t$, 
the period $\int_{C} \Omega_t$ satisfies the Picard--Fuchs equations 
\[
\left( 
\theta^4 - 5^5 t^5 (\theta + 1) (\theta +2) 
(\theta + 3) ( \theta+ 4) 
\right) \int_C \Omega_t = 0  
\]
with $\theta =  t\parfrac{}{t}$. 
The Frobenius method yields the following solution 
to this differential equation: 
\[
\Phi(t) = 
\sum_{n=0}^\infty \frac{\Gamma(1+5n + 5\epsilon)}{\Gamma(1+n+\epsilon)^5} 
t^{5n+5\epsilon} 
\]
where $\epsilon$ is an infinitesimal parameter satisfying $\epsilon^4=0$. 
Regarding $\epsilon$ as a hyperplane class on the quintic $Q$, 
we may identify the leading term of the series $\Phi(t)$ with 
the inverse Gamma class of $Q$: 
\[
\frac{\Gamma(1+5\epsilon)}{\Gamma(1+\epsilon)^5} 
= \frac{1}{\Gg_Q}. 
\]
This is how the Gamma class originally arose 
in the context of mirror symmetry \cite{HKTY, Lib99}. 
Hosono \cite{Hosono} conjectured 
(more generally for a complete intersection Calabi--Yau) 
that the period $\int_C \Omega_t$ of  
an integral 3-cycle $C\subset \ov{Y}_t$ should be written in the form 
\begin{equation} 
\label{eq:central_charge_V} 
\int_{Q} \Phi(t) \cdot \Ch(V) \cdot \Td_Q 
\end{equation} 
for a vector bundle $V \to Q$ which is ``mirror'' to $C$, 
where $\Td_Q = (2\pi\iu)^{\frac{\deg}{2}} \td_Q$. 
In physics terminology, the period $\int_C \Omega_t$ 
(or the quantity \eqref{eq:central_charge_V}) is called the \emph{central charge} 
of the D-branes $C$ (resp.~$V$). 
Hosono's conjecture has been answered affirmatively 
in \cite{Iritani11} by showing that the natural integral structure 
$H^3(\ov{Y}_t,\Z)$ agrees
with the $\Gg$-integral structure in quantum cohomology of $Q$, 
see also \cite{Horja, Borisov-Horja, Iritani09, Golyshev-Mellit}. 

A main purpose of this article is to explain a relationship 
between Gamma conjecture and mirror symmetry. 
In fact, Hosono's conjecture for a quintic $Q$ 
is closely related to the truth of Gamma conjecture 
for the ambient Fano manifold $\P^4$. A mirror of $\P^4$ is given by 
the Landau--Ginzburg model $f \colon (\C^\times)^4 \to \C$ 
and the quantum differential equation for $\P^4$ has a solution given by 
the oscillatory integral: 
\[
\int_\Gamma \exp(t f(x)) 
\frac{dx_1}{x_1} \wedge \cdots \wedge \frac{dx_4}{x_4}. 
\]
When the cycle $\Gamma$ is a Lefschetz thimble of $f(x)$ 
and $C$ is the associated vanishing cycle, this oscillatory integral 
can be written as a Laplace transform of the period $\int_C \Omega_t$. 
Correspondingly, by quantum Lefschetz principle \cite{Givental:equivariant, 
Kim:qLefschetz, YPLee:qLefschetz, Coates-Givental}, 
the quantum differential equation for $Q$ arises from   
a Laplace transform of the quantum differential equation 
for $\P^4$ \cite{Dubrovin:almostduality,IMM:qSerre}. 
Gamma conjecture relates a Lefschetz thimble of $f$ with 
an exceptional object $E$ in $D^b_{\rm coh}(\P^4)$ 
via the exponential asymptotics of the corresponding oscillatory integral; 
then the vanishing cycle $C$ associated with $\Gamma$ 
corresponds to the spherical object  $V= i^*E$ on $Q$ under 
Hosono's conjecture, see \cite[Theorem 6.9]{Iritani11}.  
See the comparison table below for quantum differential 
equations (QDE) of a Fano manifold $F$ and its anti-canonical 
section $Q$. 

\renewcommand{\arraystretch}{1.2}
\begin{table}[h] 
\begin{tabular}{c|ccc}  
& Fano  & & Calabi--Yau \\ \hline  
space $X$ & $F$ & & $Q \in |-K_F|$ \\ 
mirror & $f \colon (\C^\times)^n \to \C$ & & 
$\ov{Y}_t = \ov{f^{-1}(1/t)}$ \\
singularities of QDE  & irregular & & regular \\
solutions (central charges) & oscillatory integral of $f$ 
& $\overset{\text{Laplace}}{\leftrightarrow}$ 
& period integral of $\ov{Y}_t$ \\ 
cycles on the mirror & Lefschetz thimble & 
$\overset{\text{fiber}}{\to}$ 
& vanishing cycle \\ 
objects of $D^b_{\rm coh}(X)$ 
& exceptional object $E_i$ & 
$\overset{i^*}{\to}$ & spherical object $V_i$ \\ 
monodromy data 
& Stokes $S_{ij} = \chi(E_i,E_j)$ & &  reflection $S_{ij} - (-1)^n S_{ji}$
\end{tabular} 
\vspace{10pt} 
\caption{We expect a mirror correspondence between 
objects of $D^b_{\rm coh}(X)$ and integration cycles on 
the mirror; 
when a vanishing cycle $C$ arises as a fiber of a Lefschetz thimble $\Gamma$, 
the spherical object $V$ on $Q$ mirror to $C$ should be  
the pull-back of the exceptional object $E$ on $F$ mirror to $\Gamma$. 
A Lefschetz thimble of $f$ gives a solution to the quantum differential 
equation of $F$ which has a specific exponential asymptotics as $t\to \infty$. } 
\end{table}
\renewcommand{\arraystretch}{1} 
\vspace{-15pt} 

\subsection{Plan of the paper} 
In \S\ref{sec:qcoh_qconn}, 
we review definitions and basic facts on quantum cohomology 
and quantum connection. 
In \S\ref{sec:GI} and \S \ref{sec:GII}, we discuss Gamma conjecture I and II 
respectively. 
This part is a review of our previous paper \cite{GGI:gammagrass} 
with Vasily Golyshev. 
In \S \ref{sec:heuristics}, we give a heuristic argument which 
gives mirror oscillatory integral and the Gamma class in terms of 
polynomial loop spaces.  
In \S\ref{sec:toric} and \S \ref{sec:toric_CI}, we discuss Gamma conjecture 
for toric varieties and toric complete intersections 
using Batyrev--Borisov/Givental/Hori--Vafa mirrors. 
In \S \ref{sec:Lefschetz}, we discuss 
compatibility of Gamma conjecture I with taking hyperplane 
sections (quantum Lefschetz). 
In \S\ref{sec:Grassmannian}, we discuss Gamma conjecture 
for Grassmannians $\Gr(r,n)$ using the Hori--Vafa mirror which is 
the $r$th alternate product of the mirrors of $\P^{n-1}$. 
In Appendix \ref{app:odd}, we discuss eigenvalues of the 
quantum multiplication by $c_1(F)$ 
on odd cohomology (and on $H^{p,q}(F)$ with $p\neq q$). 

\medskip 
\noindent 
{\bf Acknowledgments.} 
We thank Vasily Golyshev for insightful discussions during the 
collaboration \cite{GGI:gammagrass}. We also thank an anonymous 
referee for very helpful comments. 
H.I.~thanks Kentaro Hori, Mauricio Romo and Kazushi Ueda for the 
discussion on the papers \cite{Benini-Cremonesi, DGLFL, Hori-Romo}. 
This project was supported by JSPS and Russian Foundation for
Basic Research under the Japan--Russia Research Cooperative Program 
``Categorical and Analytic Invariants in Algebraic Geometry''. 
S.G.~was  partially supported by the Russian Academic Excellence Project '5-100',
Dynasty Foundation, Simons IUM fellowship, and by
RFBR, research project No.15-51-50045\hspace{-10pt}{\cyr yafa}. 
H.I.~was supported by JSPS Kakenhi Grant number 25400069, 
23224002, 26610008, 24224001, 16K05127, 16H06337, 16H06335.

\section{Quantum cohomology and quantum connection} 
\label{sec:qcoh_qconn} 
Let $H^\udot(F) = H^\udot(F,\C)$ 
denote the even degree part of the Betti cohomology group of $F$
(see Appendix \ref{app:odd}). 
The (small) quantum product $\star_0$ on $H^\udot(F)$ 
is defined by the formula: 
\[
(\alpha\star_0 \beta,\gamma)_F = \sum_{d\in H_2(F,\Z)} 
\Ang{\alpha,\beta,\gamma}_{0,3,d} 
\]
where $(\cdot,\cdot)_F$ is the Poincar\'{e} pairing 
on $F$ and $\ang{\cdots}_{0,3,d}$ is the genus-zero 
three point Gromov--Witten invariants, which 
roughly speaking counts the number of rational curves in 
$F$ passing through the cycles Poincar\'{e} dual to 
$\alpha$, $\beta$ and $\gamma$ (see e.g.~\cite{Man99}). 
By the dimension axiom, these Gromov--Witten invariants 
are non-zero only if $2c_1(F) \cdot d + 2\dim F = \deg \alpha + 
\deg \beta + \deg \gamma$; since $F$ is Fano, there are 
finitely many such curve classes $d$ and the above sum 
is finite. 
The product $\star_0$ is associative and commutative, and 
$(H^\udot(F),\star_0)$ becomes a finite dimensional algebra. 
More generally we can define the big quantum product 
$\star_\tau$ for $\tau\in H^\udot(F)$: 
\[
(\alpha\star_\tau \beta,\gamma) = \sum_{d\in H_2(F,\Z)} 
\sum_{n=0}^\infty \frac{1}{n!}
\Ang{\alpha,\beta,\gamma,\tau,\dots,\tau}_{0,3+n,d}.  
\]
This defines a formal\footnote{The convergence of the 
big quantum product is not known in general.} deformation 
of the small quantum cohomology as a commutative ring. 
In this paper, we will restrict our attention to 
the small quantum product $\star_0$. 

The quantum connection is a meromorphic flat connection on the 
trivial $H^\udot(F)$-bundle over $\P^1$. It is given by: 
\begin{equation}
\label{eq:qconn}
\nabla_{z\partial_z} = 
z \parfrac{}{z} - \frac{1}{z} (c_1(F) \star_0) + \mu 
\end{equation} 
where $z$ is an inhomogeneous co-ordinate on $\P^1$ 
and $\mu \in \End(H^\udot(F))$ is defined by 
$\mu (\phi) = (p - \frac{\dim F}{2})\phi$ for $\phi\in H^{2p}(F)$. 
This is regular singular (logarithmic) at $z=\infty$ and 
irregular singular at $z=0$. 
We have a canonical fundamental solution around $z=\infty$ 
as follows: 
\begin{proposition}[\cite{Dubrovin98a,GGI:gammagrass}] 
\label{prop:fundsol} 
There exists a unique $\End(H^\udot(F))$-valued 
power series $S(z) = \id + S_1 z^{-1} + S_2 z^{-2} 
+\cdots$ which converges over the whole $z^{-1}$-plane 
such that 
\begin{align*} 
& \nabla (S(z) z^{-\mu} z^{c_1(F)} \phi) = 0 \qquad 
\forall \phi\in H^\udot(F) \\ 
& \text{$T(z) = z^{\mu}S(z) z^{-\mu}$ is regular at $z=\infty$ 
and $T(\infty) = \id$. } 
\end{align*} 
Here we set $z^{c_1(F)} = \exp(c_1(F) \log z)$ and 
$z^{-\mu} = \exp(-\mu \log z)$. 
\end{proposition} 
The fundamental solution $S(z) z^{-\mu} z^{c_1(F)}$ identifies 
the space of flat sections with the cohomology group $H^\udot(F)$. 
In other words, a basis of the cohomology group $H^\udot(F)$ 
yields a basis of flat sections via the map $S(z) z^{-\mu} z^{c_1(F)}$ --- 
this amounts to solving the quantum differential 
equation by the \emph{Frobenius method} around $z=\infty$. 

\begin{remark} 
\label{rem:big_qconn}
The quantum connection can be extended to a meromorphic flat connection 
on the trivial $H^\udot(F)$-bundle over 
$H^\udot(F) \times \P^1$ via the big quantum product: 
\begin{align}
\label{eq:big_qconn} 
\begin{split}  
\nabla_\alpha & = \partial_\alpha + \frac{1}{z} (\alpha \star_\tau) 
\qquad \alpha \in H^\udot(F), 
\\ 
\nabla_{z\partial_z} & = z \partial_z - \frac{1}{z} (E\star_\tau)   + \mu 
\end{split}
\end{align} 
where $E = c_1(F) + \sum_{i=1}^N (1- \frac{\deg \phi_i}{2}) \tau^i \phi_i$ 
with $\{\phi_i\}_{i=1}^N$ a homogeneous basis of $H^\udot(F)$ 
and $\tau = \sum_{i=1}^N \tau^i \phi_i$. 
\end{remark} 

\begin{remark}[\cite{GGI:gammagrass}]   
\label{rem:gauge}
Via the gauge transformation $z^{\mu}$ and the change $z=t^{-1}$ 
of variables, we have 
\[
z^\mu \nabla_{z\partial_z} z^{-\mu} = z\parfrac{}{z} - 
(c_1(F)\star_{- c_1(F) \log z}) 
= - \left( t \parfrac{}{t} + c_1(F)\star_{c_1(F) \log t} \right) 
\]
This gives the quantum connection along the 
``anticanonical line'' $\C c_1(F)$. 
\end{remark}

\section{Gamma conjecture I (and I')}
\label{sec:GI} 

\subsection{Property $\O$} 
We start with Property $\O$ (or Conjecture $\O$) for a Fano manifold. This is a 
supplementary condition we need to formulate Gamma conjecture I. 

\begin{definition}[Property $\O$ \cite{GGI:gammagrass}]  
\label{def:propertyO} 
Let $F$ be a Fano manifold and define a non-negative 
real number $T$ as 
\begin{equation} 
\label{eq:T} 
T := \max \{ |u|: \text{$u\in\C$ 
is an eigenvalue of $(c_1(F)\star_0)$}\} \in \ov{\Q}. 
\end{equation} 
We say that $F$ satisfies \emph{Property $\O$}  
if the following conditions are satisfied: 
\begin{enumerate} 
\item $T$ is an eigenvalue of $(c_1(F)\star_0)$ of multiplicity one; 
\item if $u$ is an eigenvalue of $(c_1(F)\star_0)$ with $|u| = T$, 
there exists a complex number $\zeta$ such that $\zeta^r =1$ and 
that $u = \zeta T$,
where $r$ is the Fano index of $F$, i.e.~the maximal 
integer $r>0$ such that $c_1(F)/r$ is an integral class. 
\end{enumerate} 
\end{definition} 

\begin{remark} 
`$\O$' means the structure sheaf of $F$. Under mirror symmetry, 
it is conjectured that the set of eigenvalues of $(c_1(F)\star_0)$ 
agrees with the set of critical values of the mirror 
Landau--Ginzburg potential $f$ 
(see, e.g.~\cite{Givental:ICM}, \cite[Theorem 6.1]{Auroux}); 
for example, this holds for Fano toric manifolds. 
Under Property $\O$, number $T$ should be a critical value 
of $f$ and the Lefschetz thimble corresponding to $T$ 
should be mirror to the structure sheaf $\O$. 
Conjecture $\O$ \cite{GGI:gammagrass} says that 
every Fano manifold satisfies Property $\O$. 
Some Fano \emph{orbifolds} with non-trivial $\pi_1^{\rm orb}(F)$ do 
not satisfy Property $\O$ \cite{Galkin:conifoldpoint, GGI:gammagrass}. 
We also note that Part (1) of Property $\O$ implies that the small 
quantum cohomology ring $(H^\udot(F),\star_0)$ has 
a field as a direct summand; this weaker conjecture is open as well. 
\end{remark} 
\begin{remark} 
Perron--Frobenius theorem says that an irreducible square matrix 
with non-negative entries 
has a positive eigenvalue with the biggest norm whose multiplicity is one. 
It is likely that $c_1(F)\star_0$ is represented by a non-negative matrix 
if we have a basis of $H^\udot(F)$ consisting of `positive' (algebraic) cycles. 
This remark is due to Kaoru Ono. 
\end{remark} 

\begin{remark} 
Property $\O$ for homogeneous spaces $G/P$ 
was recently proved by Cheong--Li \cite{Cheong-Li:conj_O}. 
\end{remark}

\subsection{Gamma class} 
\label{subsec:Gamma_class} 
The Gamma class \cite{Lib99, Lu, Iritani09} is a characteristic class defined 
for an almost complex manifold $F$. 
Let $\delta_1,\dots,\delta_n$ be the Chern roots of 
the tangent bundle $TF$ of $F$ such that 
$c(TF) = (1 + \delta_1) (1 + \delta_2)  \cdots (1+ \delta_n)$. 
The Gamma class $\Gg_F$ is defined to be 
\[
\Gg_F = \prod_{i=1}^n \Gamma(1 + \delta_i) \in H^\udot(F,\R) 
\]
where $\Gamma(z) = \int_0^\infty e^{-t} t^{z-1} dt$ is 
Euler's $\Gamma$-function. Since $\Gamma(z)$ is 
holomorphic at $z=1$, via the Taylor expansion, 
the right-hand side makes sense  
as a symmetric power series in $\delta_1,\dots,\delta_n$, 
and therefore as a (real) cohomology class of $F$. 
This is a transcendental class and is given by 
\[
\Gg_F = \exp \left( - \gamma c_1(F) + \sum_{k=2}^\infty 
(-1)^k (k-1)! \zeta(k) \ch_k(TF) \right) 
\]
where $\zeta(z)$ is the Riemann zeta function. 
As explained in the Introduction (see \S \ref{subsec:HRR}), 
the formula $(2\pi\iu)^{\frac{\deg}{2}} \td_F = e^{\pi\iu c_1(F)} 
\Gg_F \Gg_F^*$ shows that 
the Gamma class $\Gg_F$ can be regarded as a square root 
of the Todd class $\td_F$. There is also an interpretation 
of the Gamma class in terms of the free loop space $\L F$ of $F$ 
due to Lu \cite{Lu} (see also \cite{Iritani:ttstar,GGI:gammagrass}): 
$\Gg_F$ arises from the $\zeta$-function 
regularization of the $S^1$-equivariant Euler class $e_{S^1}(\N_+)$, 
where $\N_+$ is the positive normal bundle of the locus $F$ of 
constant loops in $\L F$. 

\subsection{Principal asymptotic class and Gamma conjecture I} 
\label{subsec:principal_asymptotic} 
Consider the space of flat sections for the quantum 
connection $\nabla$ \eqref{eq:qconn} 
over the positive real line $\R_{>0}$. 
We introduce the subspace $\AA$ of flat sections having 
the smallest asymptotics $\sim e^{-T/z}$ as $z\to +0$. 
\[
\AA := \left\{ s \colon \R_{>0} \to H^\udot(F) : 
\nabla s(z) = 0,\ 
\|e^{T/z} s(z)\| = O(z^{-m}) \text{ as $z\to +0$ 
($\exists m$)} 
\right\} 
\]
where $T>0$ is the number in Definition \ref{def:propertyO}. 
\begin{proposition}[{\cite[Proposition 3.3.1]{GGI:gammagrass}}]
\label{prop:dimAA_is_one}
Suppose that a Fano manifold $F$ satisfies Property $\O$. 
We have $\dim_\C \AA = 1$. Moreover, for every element 
$s(z) \in \AA$, the limit $\lim_{z\to +0} e^{T/z} s(z)$ 
exists and lies in the $T$-eigenspace $E(T)$ of 
$(c_1(F)\star_0)$. 
\end{proposition} 

The principal asymptotic class of $F$ is defined to be the class corresponding 
to a generator of the one-dimensional space $\AA$. 

\begin{definition}[\cite{GGI:gammagrass}] 
\label{def:principal_asymptotic}
Suppose that a Fano manifold $F$ satisfies Property $\O$. 
A cohomology class $A_F \in H^\udot(F)$ satisfying 
\[
\AA = \C  \left[S(z) z^{-\mu} z^{c_1(F)} A_F \right] 
\]
is called the \emph{principal asymptotic class}. 
Here $S(z) z^{-\mu} z^{c_1(F)}$ is the fundamental solution 
in Proposition \ref{prop:fundsol}. The principal asymptotic class 
is determined up to multiplication by a non-zero complex number; 
when $\ang{[\pt],A_F} \neq 0$, we can normalize 
$A_F$ so that $\ang{[\pt],A_F} =1$. 
\end{definition}

Since the space $\AA$ is identified via the asymptotics near the irregular 
singular point $z=0$ and the fundamental solution $S(z) z^{-\mu} z^{c_1(F)}$ 
is normalized at the regular singular point 
$z=\infty$, the definition of the class $A_F$ involves 
\emph{analytic continuation} along the positive real line $\R_{>0}$ 
on the $z$-plane. 
Note that $S(z) z^{-\mu} z^{c_1(F)}$ has a standard determination 
for $z\in \R_{>0}$ given by $z^{-\mu} z^{c_1(F)} 
= \exp(-\mu \log z) \exp(c_1(F) \log z)$ and $\log z\in \R$. 

\begin{conjecture}[Gamma Conjecture I \cite{GGI:gammagrass}] 
\label{conj:GammaI}
Let $F$ be a Fano manifold satisfying Property $\O$. 
The principal asymptotic class $A_F$ of $F$ is given by the 
Gamma class $\Gg_F$ of $F$. 
\end{conjecture} 

There is another description of the principal asymptotic 
class $A_F$ in terms of solutions to the quantum differential 
equation. We introduce Givental's $J$-function \cite{Givental:equivariant} 
by the formula: 
\begin{align}
\label{eq:J-function}
\begin{split}  
J_F(t) & = z^{\frac{\dim F}{2}} 
\left( S(z) z^{-\mu} z^{c_1(F)} \right)^{-1} 1 
\qquad \text{with $t=z^{-1}$} \\ 
&= e^{c_1(F) \log t} 
\left( 1 + \sum_{i=1}^N \sum_{d\in H_2(F,\Z), d\neq 0} 
\Ang{\frac{\phi_i}{1-\psi}}_{0,1,d} \phi^i t^{c_1(F) \cdot d} \right) 
\end{split} 
\end{align} 
where $S(z) z^{-\mu} z^{c_1(F)}$ is the fundamental solution 
in Proposition \ref{prop:fundsol} 
and $\psi$ is the first Chern class of the universal cotangent 
line bundle over the moduli space of stable maps. 
This is a cohomology-valued (and multi-valued) function. 
Since $S(z) z^{-\mu}z^{c_1(F)}$ is a fundamental solution 
of the quantum connection and by Remark \ref{rem:gauge}, we have 
\[
P(t,\nabla_{c_1(F)}) 1 = 0 \quad 
\Longleftrightarrow \quad 
P(t,t\textstyle\parfrac{}{t})  J_F(t) = 0 
\]
for any differential operator $P(t,t\parfrac{}{t}) \in 
\C\langle t, t\parfrac{}{t} \rangle$, where 
$\nabla_{c_1(F)} = t\parfrac{}{t} + c_1(F) \star_{c_1(F) \log t}$ 
is the quantum connection along the anticanonical line. 
In other words, $J_F(t)$ satisfies all the differential relations satisfied 
by the identity class $1$ with respect to the connection $\nabla_{c_1(F)}$; 
in this sense $J_F(t)$ is a solution of the quantum connection. 
Differential operators 
$P(t,t \parfrac{}{t})$ annihilating $J_F(t)$ are called 
\emph{quantum differential operators}. 
The principal asymptotic class $A_F$ can be computed 
by the $t\to +\infty$ asymptotics of the $J$-function: 
\begin{proposition}[\cite{GGI:gammagrass}] 
\label{prop:J_asymp}
Suppose that a Fano manifold $F$ satisfies Property $\O$ 
and let $A_F$ be the principal asymptotic class. 
Then we have an asymptotic expansion of the form: 
\[
J_F(t) = C t^{-\frac{\dim F}{2}} e^{T t} (A_F + \alpha_1 t^{-1} 
+ \alpha_2 t^{-2} + \cdots). 
\]
as $t \to +\infty$ on the positive real line, where $C\neq 0$ 
is a non-zero constant and $\alpha_i\in H^\udot(F)$. 
\end{proposition} 
\begin{proof} 
It follows from \cite[Proposition 3.6.2]{GGI:gammagrass} that 
$\lim_{t\to \infty} t^{\frac{\dim F}{2}}e^{-Tt}J_F(t)$ exists 
and is proportional to $A_F$.  
The fact that the remainder admits an asymptotic 
expansion of the form $\alpha_1 t^{-1} + \alpha_2 t^{-2}  + 
\cdots$ follows from the proof there, in particular from 
\cite[Proposition 3.2.1]{GGI:gammagrass}. 
\end{proof}

This proposition says that $\C J_F(t)$ converges to $\C A_F$ 
in the projective space $\P(H^\udot(F))$ as $t \to +\infty$. 
Since $\langle [\pt],\Gg_F\rangle =1$, we obtain the following corollary. 
\begin{corollary}[{\cite[Corollary 3.6.9]{GGI:gammagrass}}] 
\label{cor:limitGamma}
Let $F$ be a Fano manifold satisfying Property $\O$. 
Gamma conjecture I holds for $F$ if and only if we have 
\[
\Gg_F = \lim_{t\to \infty} \frac{J_F(t)}{\Ang{[\pt],J_F(t)}}. 
\]
\end{corollary} 

We can replace the continuous limit in the above corollary with 
a discrete limit of ratios of the Taylor coefficients. 
Expand the $J$-function as: 
\begin{equation} 
\label{eq:J_expansion} 
J_F(t) = e^{c_1\log t} \sum_{n=0}^\infty J_n t^n. 
\end{equation} 
Note that $J_n = 0$ if $n$ is not divisible by the Fano index $r$ 
of $F$. We have the following: 
\begin{proposition}[{\cite[Theorem 3.7.1]{GGI:gammagrass}}]
\label{prop:Apery} 
Suppose that a Fano manifold $F$ satisfies Property $\O$ 
and Gamma conjecture I. Let $r$ be the Fano index of $F$. 
Then we have 
\[
\liminf_{n\to \infty} \left|\frac{\Ang{\alpha,J_{rn}}}{\Ang{[\pt],J_{rn}}} - 
\langle \alpha, \Gg_F \rangle \right|  = 0 
\]
for every $\alpha \in H_\ldot(F)$ with $\alpha \cap c_1(F) = 0$. 
\end{proposition} 

Define the (unregularized and regularized) \emph{quantum period} of $F$ 
\cite{CCGGK}
to be 
\begin{align} 
\label{eq:qperiod}
\begin{split}
G_F(t)  & = \Ang{ [\pt], J_F(t)} = \sum_{n=0}^\infty G_n t^n \\ 
\hG_F(\kappa) & = \sum_{n=0}^\infty n! G_n \kappa^n = 
\frac{1}{\kappa} \int_0^\infty G_F(t) e^{-t/\kappa} dt 
\end{split} 
\end{align} 
where $G_n = \Ang{[\pt],J_n}$. 
It is shown in \cite[Lemma 3.7.6]{GGI:gammagrass} that 
if $F$ satisfies Property $\O$ and if $\Ang{[\pt],A_F}\neq 0$, 
the convergence radius of $\hG_F(\kappa)$ equals $1/T$. 
In particular 
\begin{equation} 
\label{eq:Cauchy-Hadamard}
\limsup_{n\to \infty} \sqrt[rn]{(rn)! |G_{rn}|} = T. 
\end{equation}
Suppose that this limit sup \eqref{eq:Cauchy-Hadamard} 
can be replaced with the limit, i.e.~$\lim_{n\to \infty} 
\sqrt[rn]{(rn)! |G_{rn}|} = T$. Then 
the argument in the proof of \cite[Theorem 3.7.1]{GGI:gammagrass} 
shows,  
under the same assumption as in Proposition \ref{prop:Apery}, that 
\begin{equation} 
\label{eq:Apery_limit}
\lim_{n\to \infty} \frac{\Ang{\alpha,J_{rn}}}{\Ang{[\pt],J_{rn}}} = 
\langle \alpha, \Gg_F \rangle 
\end{equation} 
for a class $\alpha \in H_\ldot(F)$ such that 
$\alpha \cap c_1(F) = 0$. 
Therefore we can consider the following variant of Gamma conjecture I: 
\begin{conjecture}[Gamma conjecture I'] 
\label{conj:I_dash}
For a Fano manifold $F$ and a class $\alpha \in H_\ldot(F)$ with 
$\alpha \cap c_1(F) = 0$, the limit formula \eqref{eq:Apery_limit} 
holds. 
\end{conjecture} 
\begin{remark} 
By the above discussion, Gamma conjecture I' holds if Gamma conjecture I 
holds and one has $\lim_{n\to \infty} \sqrt[rn]{(rn)! |G_{rn}|} = T$. 
\end{remark} 

The discrete limit on the left-hand side of \eqref{eq:Apery_limit} 
is called the \emph{Ap\'ery constant} (or Ap\'ery limit) 
and was studied by Almkvist--van-Straten--Zudilin \cite{AvSZ} 
in the context of Calabi--Yau differential equations 
and by Golyshev \cite{Golyshev08a} and Galkin \cite{Galkin:Apery} for Fano manifolds. 
Under Gamma conjecture I', these limits are expressed in terms of 
the zeta values 
$\zeta(2), \zeta(3), \zeta(4),\dots$. 
For some Fano manifolds, they are precisely the limits  
which Ap\'ery used to prove the irrationality 
of $\zeta(2)$ and $\zeta(3)$: 
an Ap\'ery limit for the Grassmannian $G(2,5)$ gives a fast approximation 
of $\zeta(2)$ and an Ap\'ery limit for the orthogonal Grassmannian 
$OG(5,10)$ gives a fast approximation of $\zeta(3)$ 
\cite{Galkin:Apery,Golyshev08a}. 
Most of the Ap\'ery limits of Fano manifolds are 
not fast enough to prove irrationality (see \cite{Galkin:Apery}). 
It would be very 
interesting to find a Fano manifold which gives a fast approximation 
of $\zeta(5)$, for example. 

We give a sufficient condition that ensures that the limit sup in 
\eqref{eq:Cauchy-Hadamard} can be replaced with the limit. 
When a Laurent polynomial 
$f(x) \in \C[x_1^\pm,\dots,x_m^\pm]$ is mirror to $F$, 
we expect that the quantum period $G_F(t)$ \eqref{eq:qperiod} 
for $F$ should be given by the \emph{constant term series} of $f$: 
\[
G_F(t) = \frac{1}{(2\pi\iu)^m}\int_{(S^1)^m} e^{tf(x)} 
\frac{dx_1\cdots dx_m}{x_1\cdots x_m} 
= \sum_{n=0}^\infty \frac{1}{n!} \Const(f^n) t^n 
\]
where $\Const(f^n)$ denotes the constant term 
of the Laurent polynomial $f(x)^n$. 
When this holds, $f(x)$ is said to be a \emph{weak Landau--Ginzburg 
model} of $F$ \cite{Przyjalkowski:weak}. 

\begin{lemma} 
\label{lem:limsup_lim}
Let $F$ be a Fano manifold of index $r$. 
Suppose that $F$ admits a weak Landau--Ginzburg model 
$f(x)\in \C[x_1^\pm,\dots,x_m^\pm]$ whose coefficients 
are non-negative real numbers. 
Suppose also that $\Const(f^{rn}) \neq 0$ 
for all but finitely many $n\in \Z_{\ge 0}$. Then 
the coefficients $G_n$ of the quantum period \eqref{eq:qperiod} 
are non-negative and the limit 
\[
\lim_{n\to \infty} \sqrt[rn]{(rn)! |G_{rn}|} = 
\lim_{n\to \infty} \sqrt[rn]{\Const(f^{rn})}
\]
exists. 
\end{lemma} 
\begin{proof} 
Some of the techniques here are borrowed from 
\cite{Galkin:split}. 
We set $\alpha_n = \log(\sqrt[rn]{\Const(f^{rn})})$. 
It suffices to show that $\lim_{n\to\infty} \alpha_n$ exists. 
By assumption, there exists $n_0\in \Z_{\ge 0}$ such that 
$\alpha_n$ is well-defined for all $n\ge n_0$. 
Since $\Const(f^{r(n+m)}) \ge \Const(f^{rn}) \Const(f^{rm})$, we have 
\[
\alpha_{n+m} \ge \frac{n}{n+m} \alpha_n + \frac{m}{n+m} \alpha_m 
\]
Set $\alpha := \limsup_{n\to \infty} \alpha_n$. 
For any $\epsilon>0$, there exists $n_1\ge 1$ such that 
$\alpha_{n_1}\ge \alpha - \epsilon$.  
Then we have, for all $k\ge 1$ and $0\le i<n_1$, 
\begin{align*} 
\alpha_{kn_1 + n_0+i} & \ge \frac{n_0+i}{kn_1 + n_0 + i} \alpha_{n_0+i}  
+ \frac{kn_1}{kn_1+n_0+i} \alpha_{n_1}
\end{align*} 
The right-hand side converges to $\alpha_{n_1}$ as $k\to \infty$. 
This implies that there exists $n_2 \ge n_0$ such that for all $n\ge n_2$, 
we have 
\[
\alpha_n \ge \alpha - 2\epsilon. 
\]
Since $\epsilon>0$ was arbitrary, this implies 
$\liminf_{n\to \infty} \alpha_n \ge \alpha$ and the conclusion follows. 
\end{proof}

\begin{remark} 
\label{rem:part1_propertyO}
When discussing the continuous limit, we only need to assume 
Part (1) of Property $\O$ (Definition \ref{def:propertyO}). 
More precisely, Propositions \ref{prop:dimAA_is_one}, 
\ref{prop:J_asymp} and Corollary \ref{cor:limitGamma} 
hold for Fano manifolds satisfying only Part (1) of Property $\O$, 
and Gamma conjecture I (Conjecture \ref{conj:GammaI}) 
makes sense for such Fano manifolds. 
On the other hand, 
we need Part (2) of Property $\O$ in the proof of Proposition 
\ref{prop:Apery}. 
\end{remark} 

\begin{remark}
Golyshev--Zagier \cite{Golyshev-Zagier} 
proved Gamma conjecture I for Fano threefolds of Picard rank one 
(there are 17 families of such).  
More precisely, they showed that the limit formula in Corollary \ref{cor:limitGamma} 
holds for those varieties. 
In this case, it is straightforward to check Property $\O$
using Golyshev's multiplication table \cite{Golyshev:modularity}, 
\cite[\S 5.6]{Golyshev:classification} for $(c_1(F)\star_0)$. 
For rank one Fano threefolds, 
the characteristic polynomial $\det(1- tc_1(F)\star_0)$ 
equals the symbol of the regularized quantum differential equation. 
Hence one can easily check Property $\O$ 
also from the list \cite{Fanosearch:3D_Minkowski}. 
\end{remark} 

\section{Gamma conjecture II} 
\label{sec:GII} 
The small quantum cohomology $(H^\udot(F),\star_0)$ of a Fano 
manifold $F$ is said to be semisimple if it is isomorphic to the 
direct sum of $\C$ as a ring. This is equivalent to the condition 
that $(H^\udot(F),\star_0)$ has no nilpotent elements. 
In this section, we give a refinement of Gamma conjecture I 
for a Fano manifold with semisimple quantum cohomology. 

\subsection{Formal fundamental solution}
Suppose that $(H^\udot(F),\star_0)$ is semisimple. Then 
we have an idempotent basis $\psi_1,\dots,\psi_N$ of $H^\udot(F)$ 
such that $\psi_i \star_0 \psi_j = \delta_{ij} \psi_i$. 
Define the normalized idempotent basis $\Psi_1,\dots,\Psi_N$ 
to be $\Psi_i = \psi_i/\sqrt{(\psi_i,\psi_i)}$, where $(\psi_i,\psi_i)$ 
denotes the Poincar\'e pairing of $\psi_i$ with itself. 
Note that $(\Psi_1,\dots,\Psi_N)$ 
is unique up to sign and ordering. We set 
\[
\Psi = \begin{pmatrix} 
\vert & \vert& & \vert \\ 
\Psi_1 & \Psi_2 & \cdots & \Psi_N \\ 
\vert & \vert & & \vert 
\end{pmatrix}. 
\]
This is a matrix with column vectors $\Psi_i$. We may regard it 
as a linear map $\C^N \to H^\udot(F)$. 
Let $u_1,\dots,u_N$ be the eigenvalues of $(c_1(F) \star_0)$ 
such that $c_1(F) \star_0 \Psi_i = u_i \Psi_i$.  
Let $U$ be the diagonal matrix with entries $u_1,\dots,u_N$: 
\[
U = \begin{pmatrix} 
u_1 &  & &  \\ 
& u_2 & & \\ 
& & \ddots & \\ 
&&& u_N
\end{pmatrix}
\]
The following proposition is well-known in the 
context of Frobenius manifolds. 
\begin{proposition}[{\cite[Lectures 4,5]{Dubrovin98a}, 
\cite[Theorem 8.15]{Tel10}}]
\label{prop:formal_solution} 
Suppose that the small quantum cohomology $(H^\udot(F),\star_0)$ 
is semisimple. 
The quantum connection \eqref{eq:qconn} 
near the irregular singular point $z=0$ 
admits a formal fundamental matrix solution of the form: 
\[
\Psi R(z) e^{-U/z} 
\] 
where $R(z) = \id + R_1 z + R_2 z^2 + \cdots \in \End(\C^N) [\![z]\!]$ 
is a matrix-valued formal power series. The formal solution 
$\Psi R(z) e^{-U/z}$ is 
unique up to multiplication by a signed permutation matrix 
from the right (which corresponds to the ambiguity of 
$\Psi_1,\dots,\Psi_N$). 
\end{proposition} 

\begin{remark} 
Since $R(z)$ is a formal power series in positive powers of $z$, 
the product $R(z) e^{-U/z}$ does not make sense as a power series. 
The meaning of the proposition is 
that the formal gauge transformation by $\Psi R(z)$ turns 
$\nabla_{z\partial_z}$ into $z \partial_z - U/z$. 
\end{remark} 
\begin{remark} 
We do not need to assume that $u_1,\dots,u_N$ are mutually 
distinct. 
\end{remark} 

\subsection{Lift to an analytic solution} 
\label{subsec:lift}
By choosing an angular sector, the above formal solution can be 
lifted to an actual analytic solution. This is an instance of the  
\emph{Hukuhara--Turrittin theorem} for irregular connections 
(see e.g.~\cite[Theorem 19.1]{Wasow}, 
\cite[II, 5.d]{Sabbah:Frobenius_manifold}). 
We say that a phase $\phi \in \R$ (or $e^{\iu\phi} \in S^1$) 
is \emph{admissible} for a multiset $\{u_1,u_2,\dots,u_N\}\subset \C$ 
if $\im(u_i e^{-\iu\phi}) \neq \im(u_j e^{-\iu\phi})$ for every 
pair $(u_i,u_j)$ with $u_i \neq u_j$, i.e.~$e^{\iu\phi}$ is not parallel 
to any non-zero difference $u_i - u_j$. 
\begin{proposition}[{\cite[Theorem 12.2]{Wasow}, 
\cite[Theorem A]{BJL79}, \cite[Lectures 4,5]{Dubrovin98a}, 
\cite[\S 8]{BTL}, \cite[Proposition 2.5.1]{GGI:gammagrass}}] 
\label{prop:analytic_solution} 
Let $\phi\in \R$ be an admissible phase for the spectrum 
$\{u_1,u_2,\dots,u_N\}$ of $(c_1(F)\star_0)$. 
There exist $\epsilon>0$ and an analytic fundamental solution $Y_\phi(z) 
= ( y_1^\phi(z),\dots, y_N^\phi(z))$ for the quantum connection 
\eqref{eq:qconn} on the angular 
sector $|\arg(z) - \phi| <\frac{\pi}{2} + \epsilon$ around $z=0$ such that one 
has the asymptotic expansion 
\begin{equation}
\label{eq:asymptotic_exp}
Y_\phi(z) e^{U/z} \sim \Psi R(z) 
\end{equation} 
as $z\to 0$ in the sector $|\arg(z) -\phi|<\frac{\pi}{2} + \epsilon$, 
where $\Psi R(z) e^{-U/z}$ is the formal fundamental solution 
in Proposition \ref{prop:formal_solution}. 
Such an analytic solution $Y_\phi(z)$ is unique when we 
fix the sign and the ordering of $\Psi_1,\dots,\Psi_N$. 
\end{proposition} 

\begin{remark} 
Notice that each flat section $y_i^\phi(z)$ has the exponential 
asymptotics $\sim e^{-u_i/z} \Psi_i$ as $z\to 0$ in the 
sector $|\arg (z) -\phi|<\frac{\pi}{2} + \epsilon$. 
\end{remark} 

\begin{remark} 
The precise meaning of the asymptotic expansion \eqref{eq:asymptotic_exp} 
is as follows. 
For any $0<\epsilon' < \epsilon$ and $n\in \Z_{\ge 0}$, there exists 
a constant $C= C(\epsilon',n)$ such that 
\[
\left\|Y_\phi(z) e^{-U/z} - \sum_{k=0}^n \Psi R_k z^k \right\| 
\le C |z|^{n+1} 
\] 
for all $z$ with $|\arg z - \phi| \le \frac{\pi}{2} + \epsilon'$ 
and $|z| \le 1$, where we write $R(z) = \sum_{k=0}^\infty R_k z^k$. 
\end{remark} 

\subsection{Asymptotic basis and Gamma conjecture II} 
\label{subsec:asymp_basis_GII}
Let $Y_\phi(z)= (y_1^\phi(z),\dots,y_N^\phi(z))$ be the analytic 
fundamental solution in Proposition 
\ref{prop:analytic_solution}. 
We regard $Y_\phi(z)$ as a function defined on the universal 
cover of $\C^\times$; initially it is defined on the angular 
sector $|\arg (z) - \phi|<\pi + \epsilon$ with $|z|\ll 1$, 
but can be analytically continued to the whole universal cover 
since it is a solution to a linear differential equation. 
For an admissible phase $\phi$ for $\{u_1,\dots,u_N\}$, 
we define the higher asymptotic classes $A_{F,i}^{\phi}
\in H^\udot(F)$, $i=1,\dots,N$ 
by 
\[
y_i^\phi(z) \Bigr|_{\substack{\text{parallel translate} \\ 
\text{to $\arg(z)=0$}}} 
= \frac{1}{(2\pi)^{\dim F/2}} S(z) z^{-\mu} z^{c_1(F)} A_{F,i}^\phi 
\] 
where $S(z) z^{-\mu} z^{c_1(F)}$ is the fundamental solution 
for the quantum connection in Proposition \ref{prop:fundsol}. 
We call $\{A_{F,1}^\phi,A_{F,2}^\phi,\dots,A_{F,N}^\phi\}$ 
the \emph{asymptotic basis} at the phase $\phi$. 
The asymptotic basis is the same as what Dubrovin \cite{Dubrovin98a} 
called the \emph{central connection matrix}. 

\begin{remark} 
Suppose that $F$ satisfies Property $\O$. 
When we take an admissible phase $\phi$ from 
the interval $(-\frac{\pi}{2},\frac{\pi}{2})$, 
the class $A_{F,i}^\phi$ corresponding 
to the eigenvalue $u_i = T$ is proportional to 
the principal asymptotic class $A_F$. 
\end{remark} 

\begin{remark} 
The asymptotic basis at a phase $\phi$ is unique up to 
sign and ordering. 
The sign and the ordering depend on those 
of the normalized idempotents $\Psi_1,\dots,\Psi_N$. 
Each asymptotic class $A_{F,i}^\phi$ is marked by 
the eigenvalue $u_i$ of $(c_1(F)\star_0)$. 
With respect to the pairing $[\cdot,\cdot)$ in \eqref{eq:pairing_[)}, 
these data $\{(A_{F,i}^\phi, u_i)\}_{i=1}^N$ 
form a \emph{marked reflection system} 
\cite{GGI:gammagrass}, see also Remark \ref{rem:mutation}.  
\end{remark} 

\begin{conjecture}[Gamma conjecture II \cite{GGI:gammagrass}] 
\label{conj:GammaII} 
Suppose that a Fano manifold $F$ has a semisimple small quantum 
cohomology and that $D^b_{\rm coh}(F)$ has a full exceptional 
collection. Let $\phi$ be an admissible phase for the 
spectrum $\{u_1,\dots,u_N\}$ of $(c_1(F)\star_0)$. 
We number the eigenvalues $u_1,\dots,u_N$ 
so that $\im(e^{-\iu \phi} u_1) 
\ge \im(e^{-\iu\phi} u_2) \ge \cdots \ge \im(e^{-\iu\phi}u_N)$. 
There exists a full exceptional collection $E_1^\phi,\dots,E_N^\phi$ 
such that $A_{F,i}^\phi = \Gg_F \cdot \Ch(E_i^\phi)$. 
\end{conjecture} 

Recall that $\Ch(E) = \sum_{p=0}^{\dim F} (2\pi\iu)^p 
\ch_p(E)$ is the modified Chern character. 
We stated Gamma conjecture II under a slightly restrictive assumption: as in the original 
formulation \cite[Conjecture 4.6.1]{GGI:gammagrass}, 
we can state Gamma conjecture II under the assumption that 
the big quantum cohomology is analytic and semisimple at some $\tau\in H^\udot(F)$. 
We restrict to the small quantum cohomology just for exposition. 
Note that there are Fano manifolds (symplectic isotropic Grassmannians $IG(2,2n)$) 
such that the small quantum cohomology is not semisimple, but that 
the big quantum cohomology is generically semisimple \cite{GMS,Perrin,
MPS, CMPS_K}. 

\begin{remark} 
Part (3) of Dubrovin's conjecture 
\cite[Conjecture 4.2.2]{Dubrovin98} says that the columns of 
the central connection matrix are given by $C'( \Ch(E_i))$ 
for some linear operator $C' \in \End(H^\udot(F))$ commuting 
with $c_1(F)\cup$. Gamma conjecture II says that $C ' = \Gg_F \cup$. 
Recently Dubrovin \cite{Dubrovin:Strasbourg} 
also proposed the same conjecture as Gamma conjecture II. 
\end{remark} 

\subsection{Stokes matrix}
Let $Y_\phi(z)$ and $\mY_\phi(z)$ be the analytic fundamental solutions 
from Proposition \ref{prop:analytic_solution} 
associated respectively to admissible directions $e^{\iu\phi}$ 
and $-e^{\iu\phi}$. The domains of definitions of $Y_\phi$ 
and $\mY_\phi$ are shown in Figure \ref{fig:Stokes}. 
Let $\Pi_{\pm}$ be the angular regions as in Figure \ref{fig:Stokes} 
which are components of the intersection of the domains 
of $Y_\phi$ and $\mY_\phi$. 
The \emph{Stokes matrices} are the constant matrices $S^\phi$ and 
$S^\phi_-$ satisfying 
\begin{align*} 
Y_\phi(z) & = \mY_\phi(z) S^\phi && \text{for $z\in \Pi_+$;} \\ 
Y_\phi(z) & = \mY_\phi(z) S^\phi_- && \text{for $z\in \Pi_-$}. 
\end{align*}

\begin{figure}[ht] 
\begin{picture}(300,140) 
\put(150,70){\makebox(0,0){\circle{4}}}

\put(200,70){\makebox(0,0){$Y_\phi(z)$}} 
\put(105,70){\makebox(0,0){$\mY_\phi(z)$}} 
\put(242,68){$\Longrightarrow$} 
\put(250,83){\makebox(0,0){$\phi$}} 
\put(250,55){\makebox(0,0){\tiny admissible}} 
\put(250,45){\makebox(0,0){\tiny direction}} 

\put(150,130){\makebox(0,0){$\Pi_+$}} 
\put(150,10){\makebox(0,0){$\Pi_-$}} 

\thicklines 
\path(132,0)(150,70)(132,140)
\dashline[100]{3}(168,0)(150,70)(168,140)

\end{picture} 
\caption{$\mY_\phi(z)$ is defined on the left side of the dotted line 
and $Y_\phi(z)$ is defined on the right side of the solid line.} 
\label{fig:Stokes} 
\end{figure}
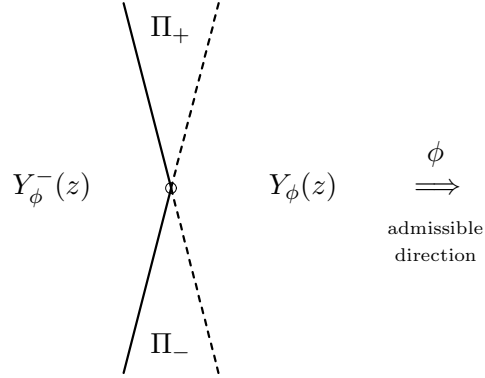
\begin{proposition}[{\cite[Theorem 4.3]{Dubrovin98a}, 
\cite[Proposition 2.6.4]{GGI:gammagrass}}]  
Let $Y_\phi(z) = (y_1^\phi(z),\dots, y_N^\phi(z))$ 
be the fundamental solution from Proposition \ref{prop:analytic_solution}. 
The Stokes matrices at phase $\phi$ 
are given by $S^\phi_{ij} = S^\phi_{-,ji} = (y_i^\phi(e^{-\pi\iu} z), 
y_j^\phi(z))$, where 
$y_i^{\phi}(e^{-\pi\iu} z)$ denotes the analytic continuation 
of $y_i(z)$ along the path $[0,\pi] \ni \theta \mapsto e^{-\iu\theta} z$. 
The flat sections $y_1^\phi,\dots,y_N^\phi$ 
are semi-orthogonal in the following sense: 
\[
S^\phi_{ij} = 
\begin{cases} 
0 & \text{if $(i\neq j$ and $u_i=u_j)$   
or 
$\im(e^{-\iu \phi} u_i) < \im(e^{-\iu \phi} u_j)$}; \\ 
1 & \text{if $i=j$}. 
\end{cases}
\] 
\end{proposition} 

Using the non-symmetric pairing $[\cdot,\cdot)$ given in 
\eqref{eq:pairing_[)}, we have 
\begin{align*} 
S^\phi_{ij} & = (y_i^\phi(e^{-\iu\pi} z), y_j^\phi(z) )  \\ 
& = \frac{1}{(2\pi)^{\dim F}} \left(S(-z) z^{-\mu}  
 e^{\pi\iu\mu} z^{c_1(F)}e^{-\pi\iu c_1(F)} A_{F,i}^\phi, 
S(z) z^{-\mu} z^{c_1(F)} A_{F,j}^\phi \right) \\
& = [A_{F,i}^\phi, A_{F,j}^\phi)   
\end{align*} 
where we used the fact that $(S(-z) \alpha, S(z) \beta )  = (\alpha,\beta)$ 
and $(z^{-\mu}\alpha,z^{-\mu}\beta) = (\alpha,\beta)$ 
(see \cite{Dubrovin98a}). 
Therefore, the factorization \eqref{eq:factorize_HRR} of the 
Hirzebruch--Riemann--Roch formula implies the 
following corollary. 

\begin{corollary} 
Suppose that a Fano manifold $F$ satisfies Gamma conjecture II. 
Then there exists a full exceptional collection $E_1,\dots, E_N$ 
of $D^b_{\rm coh}(F)$ such that $\chi(E_i,E_j)$ equals the 
Stokes matrix $S_{ij}$.  (This conclusion is part (2) of 
Dubrovin's conjecture \cite[Conjecture 4.2.2]{Dubrovin98}). 
\end{corollary} 

\begin{remark} 
\label{rem:mutation} 
The asymptotic basis can be defined similarly for the big quantum 
product $\star_\tau$ with $\tau \in H^\udot(F)$ as far as 
$\star_\tau$ is semisimple. 
We have an asymptotic basis $A_{F,i}^{\phi,\tau}$ 
depending on both $\phi$ and $\tau$, and Gamma conjecture II 
makes sense at general $\tau\in H^\udot(F)$. 
The truth of the Gamma conjecture II is, however, independent of 
the choice of $(\tau,\phi)$. This is because the asymptotic basis 
changes (discontinuously) by mutation as $(\tau,\phi)$ varies, 
and we can consider the corresponding mutation for exceptional 
collections. The right mutation 
of asymptotic bases takes the form 
\[
(A_1, A_2, \dots, \overset{i}{A_i}, 
\overset{i+1}{A_{i+1}}, \dots A_N) \mapsto 
(A_1,A_2,\dots, \overset{i}{A_{i+1}}, 
\overset{i+1}{A_i -  [A_i, A_{i+1}) A_{i+1}}, 
\cdots, A_N) 
\]
if we order the asymptotic basis so that 
$\im(e^{-\iu \phi} u_1)\ge \im(e^{-\iu \phi} u_2) 
\ge \cdots \ge \im(e^{-\iu\phi} u_N)$. 
The right mutation happens when the eigenvalue $u_{i+1}$ crosses 
the ray $u_i + \R_{\ge 0} e^{\iu\phi}$. 
(The left mutation is the inverse of the right mutation). 
The braid group action on asymptotic bases and Stokes matrices 
has been studied by Dubrovin, see \cite{Dubrovin98a, GGI:gammagrass} 
for more details. 
\end{remark}

\begin{remark} 
In general, the quantum connection of a smooth projective 
variety (or a projective orbifold) $X$  
is underlain by the integral local system consisting of flat sections 
$(2\pi)^{-\frac{\dim F}{2}} 
S(z) z^{-\mu}z^{c_1(F)} \Gg_X \Ch(E)$ with $E\in K^0(X)$.  
This is called the \emph{$\Gg$-integral structure} \cite{Iritani09,KKP08}. 
In this language, Gamma conjecture implies that this integral structure 
should be compatible with the Stokes structure 
at the irregular singular point $z=0$ of the quantum connection. 
Katzarkov-Kontsevich-Pantev \cite{KKP08} imposed the compatibility 
of rational structure with Stokes structure as part of 
conditions for nc-Hodge structure. 
Our Gamma conjecture can be viewed as an adaptation of their 
compatibility condition to the quantum cohomology of Fano manifolds. 
Note that Gamma conjecture II implies the integrality of the 
Stokes matrix $S_{ij}$. 
\end{remark} 

\begin{remark} 
Recently, Sanda--Shamoto \cite{Sanda-Shamoto} proposed a generalization 
of Gamma Conjecture II to the non-semisimple case, and called it 
Dubrovin type conjecture.  
(Among the authors of \cite{GGI:gammagrass}, 
such a conjecture was locally called ``Gamma Conjecture III'', 
although we did not have a precise formulation.) 
\end{remark} 

\section{Mirror heuristics} 
\label{sec:heuristics} 
In this section we give a heuristic argument 
which gives the mirror oscillatory integral 
and the Gamma class 
from the polynomial loop space (quasi map space) 
for $\P^{N-1}$. This is motivated by 
Givental's equivariant Floer theory heuristics 
\cite{Givental:ICM, Givental:homological}. 
The argument in this section can be applied more generally 
to toric varieties to yield their mirrors and Gamma classes. 

\subsection{Polynomial loop space} 
Givental \cite{Givental:ICM,Givental:homological} conjectured that the quantum 
$D$-module (i.e.~quantum connection) should be identified with 
the $S^1$-equivariant Floer theory for the free loop space 
(see also \cite{Vlassopoulos, Iritani:EFC, Arkhipov-Kapranov}). 
Following Givental, we consider an algebraic version 
of the loop space instead 
of the actual free loop space. 
The (Laurent) \emph{polynomial loop space} of $\P^{N-1}$ 
is defined to be 
\[
L_{\rm poly} \P^{N-1} = 
\left( \C[\zeta,\zeta^{-1}]^N \setminus \{0\}\right)/\C^\times.  
\]
where $\C^\times$ acts on $\C[\zeta,\zeta^{-1}]^N$ by 
scalar multiplication. 
The polynomial loop space $L_{\rm poly}\P^{N-1}$ 
can be also described as a symplectic 
reduction of $\C[\zeta,\zeta^{-1}]^N \cong \C^{\oplus \infty}$ 
by the diagonal $S^1$-action. 
For a point $(a_1(\zeta),\dots, a_N(\zeta)) \in \C[\zeta,\zeta^{-1}]^N$, 
we write $a_i(\zeta) = \sum_{n\in \Z} a_{i,n} \zeta^n$ and regard 
$(a_{i,n} : 1\le i\le N, n \in \Z)$ as a co-ordinate system on 
$\C[\zeta,\zeta^{-1}]^N$. 
With respect to the standard K\"ahler form 
\begin{equation} 
\label{eq:Kahlerform_Cinfinity}
\omega = \frac{\iu}{2} \sum_{i=1}^N \sum_{n\in \Z} 
d a_{i,n} \wedge  d\ov{a_{i,n}}
\end{equation}
on $\C[\zeta,\zeta^{-1}]^N$, 
the diagonal $S^1$-action admits a moment map\footnote
{We identify $\Lie(S^1)^*$ with $\R$ so that the radian 
angular form $d\theta$ on $S^1=\{e^{\iu\theta} : \theta \in \R\}$ 
corresponds to $2\pi$. With this choice, the reduced symplectic form 
$\omega_u$ on $\mu^{-1}(u)/S^1$ 
represents an integral cohomology class precisely 
when $u\in\Z$. 
} 
$\mu \colon \C[\zeta,\zeta^{-1}]^N \to \Lie(S^1)^* \cong \R$ 
given by 
\[
\mu(a_1(\zeta),\dots,a_N(\zeta)) = 
\pi \sum_{i=1}^N \sum_{n\in \Z} |a_{i,n}|^2. 
\]
This satisfies $\iota_{X} \omega + d \mu = 0$ for the 
vector field  
$X = 2\pi \iu \sum_{i=1}^N \sum_{n\in \Z} (a_{i,n} \parfrac{}{a_{i,n}} 
- \ov{a_{i,n}}\parfrac{}{\ov{a_{i,n}}})$ generating the 
$S^1$-action. 
Then we have 
\[
L_{\rm poly} \P^{N-1} \cong \mu^{-1}(u) /S^1 
\]
for every $u\in\R_{>0}$. 
Via the symplectic reduction, $L_{\rm poly} \P^{N-1}$ 
is equipped with the reduced symplectic form $\omega_u$ 
such that the pull-back of $\omega_u$ to $\mu^{-1}(u)$ 
equals the restriction $\omega|_{\mu^{-1}(u)}$. 
The class $\omega_u$ represents the cohomology class $u c_1(\O(1))$ 
on $L_{\rm poly} \P^{N-1}$. 

The loop rotation defines the $S^1$-action on 
$L_{\rm poly} \P^{N-1}$ given by 
$[a_1(\zeta),\dots, a_N(\zeta)] 
\mapsto [a_1(\lambda \zeta),\dots, a_N(\lambda \zeta)]$ 
with $\lambda \in S^1$. 
With respect to the reduced symplectic form $\omega_u$ 
on $L_{\rm poly} \P^{N-1}$, 
this $S^1$-action admits a moment map $H_u$ given by 
\[
H_u([a_1(\zeta),\dots,a_N(\zeta)]) = \pi 
\sum_{i=1}^N \sum_{n\in \Z} n |a_{i,n}|^2 
\qquad \text{with $(a_1(\zeta),\dots,a_N(\zeta))\in \mu^{-1}(u)$}.  
\]
The function $H_u$ is an analogue of the \emph{action functional}  
on the free loop space. 

\begin{remark} 
More precisely, the polynomial loop space should be regarded 
as an analogue of the \emph{universal cover} 
of the free loop space $\L \P^{N-1}$. In fact, we have an analogue of 
the deck transformation on $L_{\rm poly} \P^{N-1}$ 
given by $[a_1(\zeta),\dots, a_N(\zeta)] 
\mapsto [\zeta a_1(\zeta),\dots, \zeta a_N(\zeta)]$; 
this corresponds to a generator of $\pi_1(\L \P^{N-1}) 
\cong \pi_2(\P^{N-1}) = \Z$. 
Recall that the action functional is defined on the 
universal cover of $\L \P^{N-1}$. 
\end{remark} 

\subsection{Solution as a path integral}
Recall that 
symplectic Floer theory is an infinite-dimensional analogue of 
the Morse theory with respect to the action functional on 
the loop space. We consider the Morse theory on $L_{\rm poly} \P^{N-1}$ 
with respect to the Bott--Morse function $H_u$. 
Note that the critical set of $H_u$ is a disjoint union of infinite copies of 
$\P^{N-1}$ given by $(\P^{N-1})_n = 
\{ [a_1 \zeta^n,\dots,a_N \zeta^n ] \} \subset L_{\rm poly} \P^{N-1}$ 
for each $n\in \Z$. 
The Floer fundamental cycle $\Delta$ is defined to be the closure 
of the stable manifold associated to the critical component $(\P^{N-1})_0$. 
This is given by 
\[
\Delta = \left\{ [a_1(\zeta),\dots, a_N(\zeta)] \in L_{\rm poly} \P^{N-1} : 
a_i(\zeta) \in \C[\zeta] \right \}. 
\]
Under the isomorphism between the Floer homology and the quantum 
cohomology \cite{PSS}, the Floer fundamental cycle corresponds 
to the identity class $1\in H^\udot(\P^{N-1})$. 
Consider the following equivariant 2-form on $L_{\rm poly} \P^{N-1}$ 
(in the Cartan model) 
\[
\Omega_u = \omega_u - z H_u 
\]
where $z \in H^2_{S^1}(\pt,\Z)$ 
is a positive generator. 
Note that $\Omega$ is equivariantly closed since 
$H_u$ is a Hamiltonian for the $S^1$-action. 
Givental \cite{Givental:ICM,Givental:homological} proposed that 
the infinite-dimensional integral 
(which could be viewed as a Feynman path integral) 
\begin{equation} 
\label{eq:path_integral}
\int_\Delta e^{\Omega_u/z} 
\end{equation} 
should give a solution to the quantum differential equation for 
$\P^{N-1}$. 
This may be viewed as the image of the Floer fundamental cycle $\Delta$ under 
the homomorphism $C\mapsto \int_{C} e^{\Omega_u/z}$. 
The integral does not have a rigorous definition in 
mathematics, but we can heuristically 
compute this quantity in two different ways --- 
one is by a direct computation and the other is by localization. 
The former method yields a mirror oscillatory integral 
and the latter (due to Givental \cite{Givental:ICM,Givental:homological}) 
yields the $J$-function of $\P^{N-1}$. 
In Givental's original calculation, however, 
an infinite (constant) factor corresponding 
to the Gamma class has been ignored. 
From this calculation we obtain a (mirror) integral representation 
of the $\Gg_{\P^{N-1}}$-component of the $J$-function. 

\subsection{Direct calculation} 
\label{subsec:direct}
We compute the infinite dimensional integral \eqref{eq:path_integral} directly. 
We regard $z$ as a positive real parameter. 
Since $\Delta = (\mu^{-1}( u) 
\cap \C[\zeta]^N) / S^1$, we have 
\[
\eqref{eq:path_integral} 
= \int_{(\mu^{-1}(u) \cap \C[\zeta]^N)/S^1} 
e^{-H_u} e^{\omega_u/z} 
= \int_{\mu^{-1}(u) \cap \C[\zeta]^N} e^{-H_u} \frac{d\theta}{2\pi} 
\wedge e^{\omega/z}   
\]
where $d\theta$ is the angular form (connection form) on 
the principal $S^1$-bundle $\mu^{-1}(u) \to \mu^{-1}(u)/S^1$ given 
by $d\theta = \frac{\pi}{2 \mu \iu}\sum_{i=1}^N \sum_{k\in \Z} 
(\ov{a_{i,n}} da_{i,n} - a_{i,n} d\ov{a_{i,n}})$ 
(satisfying $d\theta(X) = 2\pi$) 
and $\omega$ is the K\"ahler form \eqref{eq:Kahlerform_Cinfinity}. 
Changing co-ordinates $a_{i,n} \to \sqrt{z}a_{i,n}$, we find that 
this equals 
\[
\int_{\mu^{-1}(u/z) \cap \C[\zeta]^N} 
e^{-z H_u} \frac{d\theta}{2\pi} \wedge e^\omega. 
\]
If $\C[\zeta]^N$ were a finite dimensional vector space, 
the top-degree component of the differential form 
$d\mu \wedge \frac{d\theta}{2\pi} \wedge e^{\omega}$ would equal 
the Liouville volume form on $\C[\zeta]^N$ associated to $\omega$. 
Therefore we may write this as 
\[
\int_{\C[\zeta]^N} \delta(\mu - u/z) e^{-zH_u} d \vol
\]
where $d \vol = \bigwedge_{i=1}^N \bigwedge_{n=0}^\infty 
\left(\frac{\iu}{2} da_{i,n} \wedge d\ov{a_{i,n}}\right)$ 
and $\delta(x)$ is the Dirac delta-function. 
Using $\delta(x) = \frac{1}{2\pi} \int_{-\infty}^\infty 
e^{\iu x \xi} d\xi$, we can compute this as: 
\begin{align}
\nonumber
\frac{1}{2\pi} 
\int_{\C[\zeta]^N} d\vol \int_{-\infty}^\infty & 
d\xi e^{\iu \xi (\mu - u/z)} 
e^{-zH_u}  \\ 
\nonumber
& = \frac{1}{2\pi} 
\int_{-\infty}^\infty d\xi 
e^{-\iu \xi u/z} 
\int_{\C[\zeta]^N} 
d \vol 
\prod_{i=1}^N \prod_{n=0}^\infty 
e^{- \pi |a_{i,n}|^2(n z - \iu \xi )} \\
\label{eq:before_zeta_reg}
& = \frac{1}{2\pi} \int_{-\infty}^\infty d\xi 
e^{-\iu \xi u/z} \prod_{i=1}^N \prod_{n=0}^\infty 
\frac{1}{n z - \iu \xi} 
\end{align}
By the $\zeta$-function regularization, we can regularize 
the infinite product to get: 
\[
\frac{1}{\prod_{n=0}^\infty (n z - \iu \xi)} 
\sim (2\pi z)^{-1/2} z^{-\iu \xi/z}
\Gamma(-\iu \xi/z). 
\]
Therefore \eqref{eq:before_zeta_reg} should equal, after 
the change $\xi \to z \xi$ of co-ordinates, 
\begin{equation} 
\label{eq:Gamma_Mellin_trans}
\frac{z}{2\pi (2\pi z)^{N/2}} 
\int_{-\infty}^\infty d\xi 
e^{-\iu (u + N\log z)\xi} \Gamma(-\iu \xi)^N.  
\end{equation} 
This integral makes sense if we perturb the integration 
contour so that it avoids the singularity at 
$\xi = 0$. We will consider the perturbed 
contour from $-\infty + \iu \epsilon$ to $\infty +\iu \epsilon$ 
with $\epsilon>0$. Then the integral \eqref{eq:Gamma_Mellin_trans} 
is just a (finite-dimensional) Fourier transform 
and the following discussion can be made completely rigorous. 
Using the integral representation $\Gamma(z) = \int_0^\infty 
e^{-x} z^{x-1} dx$ of the $\Gamma$-function, 
we find that \eqref{eq:Gamma_Mellin_trans} equals 
\begin{align}
\nonumber  
\frac{z}{2\pi (2\pi z)^{N/2}} &
\int_{-\infty}^\infty d\xi \int_{[0,\infty)^N} 
\frac{dx_1}{x_1} \cdots \frac{dx_N}{x_N}  
e^{-\iu \xi (u+N \log z+\sum_{i=1}^N \log x_i)} 
e^{-(x_1+ \cdots + x_N)} \\ 
\nonumber 
& = \frac{z}{(2\pi z)^{N/2}}
\int_{[0,\infty)^N} 
\frac{dx_1}{x_1} \cdots \frac{dx_N}{x_N}  
 \delta(u + N\log z+ \textstyle\sum_{i=1}^N \log x_i) 
e^{-(x_1+ \cdots + x_N)} \\ 
\label{eq:direct_cal} 
& = \frac{z}{(2\pi z)^{N/2}} 
\int_{[0,\infty)^{N-1}} 
\frac{dx_1}{x_1} \cdots \frac{dx_{N-1}}{x_{N-1}}  
e^{-\left( x_1+ \cdots + x_{N-1} + 
\frac{e^{-u}}{x_1 \cdots x_{N-1}}\right)/z} 
\end{align} 
where in the last line we considered the co-ordinate change
$x_i \to x_i/z$. The Landau--Ginzburg mirror of $\P^{N-1}$ is given 
by the Laurent polynomial function $f(x_1,\dots,x_{N-1}) = 
x_1 + \cdots + x_{N-1} + \frac{e^{-u}}{x_1\cdots x_{N-1}}$ 
and this is the associated oscillatory integral 
\cite{Givental:ICM,Hori-Vafa}. 

\subsection{Calculation by localization} 
\label{subsec:localization}
Next we calculate the quantity \eqref{eq:path_integral} 
using the localization formula of equivariant cohomology 
(or Duistermaat--Heckman formula) 
\cite{Duistermaat-Heckman, Berline-Vergne, Atiyah-Bott}. 
The $S^1$-fixed set in $\Delta$ is the disjoint union of 
$(\P^{N-1})_n$ with $n\ge 0$ 
and we have $[\Omega_u]|_{(\P^{N-1})_n} 
= u h - z n u$, where $h := c_1(\O(1)) 
\in H^2(\P^{N-1})$ is the hyperplane class. 
Therefore 
\[
\int_{\Delta} e^{\Omega_u/z} 
= \sum_{n=0}^\infty \int_{(\P^{N-1})_n} 
\frac{e^{uh/z - nu}}{e_{S^1}(\N_n)} 
\] 
where $\N_n$ is the (infinite-rank) normal bundle of $(\P^{N-1})_n$ in $\Delta$ 
and 
\[
e_{S^1}(\N_n) = \prod_{k\ge -n, k \neq 0} (h + kz)^N 
= \prod_{k=1}^\infty (h+ kz)^N \prod_{k=1}^n (h-kz)^N.  
\]
The infinite factor $\prod_{k=1}^\infty (h+ kz)^N$ was discarded 
in the original calculation of Givental \cite{Givental:homological}. 
Using again the $\zeta$-function regularization, we find that this 
factor yields the Gamma class:  
\[
\frac{1}{\prod_{k=1}^\infty (h + kz)} \sim 
\left( \frac{z}{2\pi}\right)^{1/2} z^{h/z} \Gamma(1+h/z). 
\]
Thus we should have
\begin{align*} 
\int_\Delta e^{\Omega_u/z} 
& = \sum_{n=0}^\infty \int_{\P^{N-1}} 
\left( \frac{z}{2\pi} \right)^{N/2} 
z^{Nh/z} \Gamma(1+h/z)^N 
\frac{e^{uh/z -nu}}{\prod_{k=1}^n (h-kz)^N}  \\ 
& = \frac{z}{(2\pi z)^{N/2}}
\int_{\P^{N-1}} \Gamma(1+h)^N \sum_{n=0}^\infty 
\frac{(e^u z^N)^{h-n}}{\prod_{k=1}^n (h-k)^N}.  
\end{align*} 
Recall that the $J$-function \eqref{eq:J-function} of $\P^{N-1}$ is given by 
\cite{Givental:equivariant}
\begin{equation} 
\label{eq:J-function_P} 
J_{\P^{N-1}} (t) = 
\sum_{n=0}^\infty \frac{t^{N (h+n)}}{\prod_{k=1}^n (h + k )^N}.   
\end{equation} 
Thus, using $\Gg_{\P^{N-1}}=\Gamma(1+h)^N$, we obtain 
\begin{equation} 
\label{eq:loc_cal} 
\int_{\Delta} e^{\Omega_u/z} = \frac{z}{(2\pi z)^{N/2}} 
(2\pi \iu)^{N-1} \left[ J_{\P^{N-1}}(e^{\pi\iu} t), \Gg_{\P^{N-1}} \right) 
\end{equation} 
under the identification $t = e^{-u/N} z^{-1}$, where 
$[\cdot,\cdot)$ is the non-symmetric pairing defined 
in \eqref{eq:pairing_[)}. 

\begin{remark}
\label{rem:centralcharges}  
The quantity \eqref{eq:loc_cal} coincides, up to a factor, 
with the \emph{quantum cohomology central charge} 
of $\O_{\P^{N-1}}$ \cite{Iritani09,GGI:gammagrass}. 
For a vector bundle $E$ on a Fano manifold $F$, the quantum cohomology 
central charge $Z(E)$ is defined to be: 
\begin{align*} 
Z(E) & = (2\pi\iu)^{\dim F} \left[J_F(e^{\pi\iu} t), \Gg_F \Ch(E)\right) \\ 
& = z^{\frac{\dim F}{2}} \left(1, S(z) z^{-\mu}z^{c_1(F)} \Gg_F \Ch(E)
\right). 
\end{align*} 
where $t = z^{-1}$. 
\end{remark} 

\subsection{Comparison} 
We computed the infinite dimensional integral \eqref{eq:path_integral} 
in two ways. 
Comparing \eqref{eq:direct_cal} and \eqref{eq:loc_cal}, we should have 
the equality 
\begin{equation} 
\label{eq:oscint_centralcharge}
\int_{[0,\infty)^{N-1}} 
e^{-\left( x_1+ \cdots + x_{N-1} + 
\frac{e^{-u}}{x_1 \cdots x_{N-1}}\right)/z} 
\frac{dx_1}{x_1} \cdots \frac{dx_{N-1}}{x_{N-1}}  
= (2\pi\iu)^{N-1} \left[J_{\P^{N-1}}(e^{\pi\iu} t), \Gg_{\P^{N-1}}\right) 
\end{equation}
with $t = e^{-u/N} z^{-1}$. 
This oscillatory integral representation 
yields the asymptotic expansion (as $t\to +\infty$) 
\[
Z(\O_{\P^{N-1}}) = (2\pi\iu)^{N-1} 
\left[J_{\P^{N-1}}(e^{\pi\iu} t), \Gg_{\P^{N-1}}\right) 
\sim \text{const}\times t^{-\frac{N-1}{2}} e^{-Nt} 
\]
which can be used to prove the Gamma conjecture for $\P^{N-1}$. 
See \S\ref{sec:toric} and 
\cite[\S 3.8]{GGI:gammagrass}. 

\begin{remark} 
We have a rigorous independent proof of 
the equality \eqref{eq:oscint_centralcharge} 
(see \cite{Iritani09,KKP08}). 
Recall that we can write the left-hand side 
as the Fourier transform \eqref{eq:Gamma_Mellin_trans} 
of $\Gamma(-\iu\xi)^N$; by closing the integration contour 
in the lower half $\xi$-plane and writing the integral as the sum of 
residues at $\xi = 0, -\iu,-2\iu,-3\iu,\dots$, we arrive at the expression 
on the right-hand side. 
\end{remark} 

\begin{remark} 
A similar regularization of an infinite dimensional integral 
appears in the computation of (sphere or hemisphere) 
partition functions of (2,2) supersymmetric gauge theories, 
see Benini-Cremonesi \cite{Benini-Cremonesi}, 
Doroud--Gomis--Le-Floch--Lee \cite{DGLFL} and 
Hori-Romo \cite{Hori-Romo}. It appears that 
the computations in \S \ref{subsec:direct} and \S\ref{subsec:localization} 
correspond, in the terminology of \cite{Benini-Cremonesi,DGLFL}, 
to the localization on the 
\emph{Coulomb branch} 
and on the \emph{Higgs branch} respectively. 
\end{remark} 

\section{Toric Manifold} 
\label{sec:toric}
In this section, we discuss Gamma conjecture for Fano toric manifolds. 
We prove Gamma conjecture I by assuming 
a certain condition for the mirror Laurent polynomial $f$ 
which is analogous to Property $\O$. 

Let $X$ be an $n$-dimensional Fano toric manifold. 
A mirror of $X$ is given by the Laurent polynomial 
\cite{Givental:ICM, Hori-Vafa, Givental:toric}:  
\[
f(x) = x^{b_1} + x^{b_2} + \cdots + x^{b_m}  
\]
where $x=(x_1,\dots,x_n) \in (\C^\times)^n$, 
$b_1,\dots,b_m \in \Z^n$ are primitive generators of 
the 1-dimensional cones of the fan of $X$ 
and $x^{b_i} = x_1^{b_{i1}} \cdots x_n^{b_{in}}$ 
for $b_i = (b_{i1},\dots,b_{in})$. 
By mirror symmetry \cite{Givental:toric, Iritani09}, 
the small quantum cohomology ring $(H^\udot(X),\star_0)$ is 
isomorphic to the Jacobian ring $\C[x_1^\pm,\dots,x_n^\pm]/
(\partial_{\log x_1} f,\dots, \partial_{\log x_n} f)$ and 
the first Chern class $c_1(X)$ corresponds to the class of $f$ 
in the Jacobian ring under this isomorphism. 
Therefore, the set of eigenvalues of $(c_1(X) \star_0)$ 
coincides with the set of critical values of $f$; moreover their multiplicities 
also coincide. 
The restriction of $f$ to the real locus $(\R_{>0})^n$ is 
strictly convex since the logarithmic Hessian 
\[
\parfrac{^2f}{\log x_i \partial \log x_j}(x) 
=\sum_{k=1}^m b_{kj} b_{ki} x^{b_k}
\]
is positive definite for any $x\in (\R_{>0})^n$. 
One can also show that $f|_{(\R_{>0})^n}$ is proper and bounded from below 
since the convex hull of $b_1,\dots,b_m$ contains the origin 
in its interior (cf.~Remark \ref{rem:ineq}). 
Therefore 
$f|_{(\R_{>0})^n}$ admits a global minimum at a unique 
point $x_{\rm con}\in (\R_{>0})^n$. We call $x_{\rm con}$ 
the \emph{conifold point} \cite{Galkin:conifoldpoint,GGI:gammagrass}.  
Consider the following condition: 
\begin{condition}[analogue of Property $\O$ for toric manifolds] 
\label{cond:toric_O} 
Let $X$ be a Fano toric manifold and $f$ be its mirror Laurent 
polynomial. 
Let $T_{\rm con}=f(x_{\rm con})$ be the value of $f$ at the conifold point  
$x_{\rm con}$. One has 
\begin{itemize} 
\item[(a)] every critical value $u$ of $f$ satisfies $|u|\le T_{\rm con}$;  
\item[(b)] the conifold point is the unique critical point of $f$ 
contained in $f^{-1}(T_{\rm con})$.   
\end{itemize} 
\end{condition} 

From the relationship between the quantum cohomology of $X$ 
and the Jacobian ring of $f$ as above, we know that 
Condition \ref{cond:toric_O} implies\footnote
{Note that Condition \ref{cond:toric_O} means slightly more than 
what Part (1) of Property $\O$ (Definition \ref{def:propertyO}) says 
for Fano toric manifolds;  
Condition \ref{cond:toric_O} additionally claims that $T= T_{\rm con}$.} 
Part (1) of Property $\O$ (Definition \ref{def:propertyO}). 
Recall that Part (1) was enough to make sense of Gamma Conjecture I 
(see Remark \ref{rem:part1_propertyO}). 

\begin{remark} 
One can define the conifold point for every Laurent polynomial 
such that the Newton polytope contains the origin in its interior 
and that all the coefficients are positive, but 
Condition \ref{cond:toric_O} does not always hold. 
For instance, the one-dimensional Laurent polynomial 
$f(x) = x^{-1} + x + t x^2$ with a sufficiently 
small $t>0$ does not satisfy Part (a) of the condition. 
Also, if the lattice generated by $b_1,\dots,b_m$ is not equal 
to $\Z^n$, then $f(x) = x^{b_1} + x^{b_2} + \cdots + x^{b_m}$ 
has non-trivial diagonal symmetry $\{\zeta \in (\C^\times)^n: 
f(x)  = f(\zeta x) \} \neq \{1\}$ and Part (b) of the condition fails.  
\end{remark} 

\begin{theorem}
\label{thm:toric_GammaI}
Suppose that a Fano toric manifold $X$ satisfies Condition \ref{cond:toric_O}. 
Then $X$ satisfies Gamma Conjecture I. 
\end{theorem} 
\begin{proof} 
It follows from the argument in \cite[\S 4.3.1]{Iritani09} that 
\[
z^{n/2}\left(\phi, S(z) z^{-\mu} z^{c_1(X)} \Gg_X \right)_X =
\int_{(\R_{>0})^n} e^{-f(x)/z} \varphi(x,z) 
\frac{dx_1\cdots dx_n}{x_1\cdots x_n} 
\]
for $z>0$, where $\phi\in H^\udot(X)$ and 
$\varphi(x,z)\in \C[x_1^\pm,\dots,x_n^\pm,z]$ 
is such that the class $[\varphi(x,z) e^{-f(x)/z} dx_1
\cdots dx_n/(x_1\cdots x_n)]$ corresponds to 
$\phi$ under the mirror isomorphism in \cite[Proposition 4.8]{Iritani09} 
and $n= \dim X$. When $\varphi=1$, one has $\phi=1$ and 
this gives the integral representation 
\[
Z(\O_X) = \int_{(\R_{>0})^n} e^{-f(x)/z} 
\frac{dx_1\cdots dx_n}{x_1 \cdots x_n}
\]
of the quantum cohomology central charge $Z(\O_X)$ 
from Remark \ref{rem:centralcharges}. 
We already saw this in \eqref{eq:oscint_centralcharge} for $X = \P^{N-1}$. 
This integral representation implies that the flat section 
$S(z)z^{-\mu}z^{c_1(X)}\Gg_X$ has the following asymptotics as 
$z\to +0$, i.e.~  
\[
\left\|e^{T_{\rm con}/z} S(z) z^{-\mu} z^{c_1(X)} \Gg_X\right\| 
= O(1) 
\] 
where recall that $T_{\rm con}$ is the global minimum of 
$f$ on $(\R_{>0})^n$. 
Condition \ref{cond:toric_O} ensures that $T=T_{\rm con}$ and 
hence 
the conclusion follows. 
\end{proof} 

A toric manifold has a generically semisimple quantum 
cohomology. 
We note that the following weaker version of 
Gamma Conjecture II (Conjecture \ref{conj:GammaII}) 
can be shown for a toric manifold. 

\begin{theorem}[\cite{Iritani09}]
\label{thm:toric_GammaII}  
Let $X$ be a Fano toric manifold. 
We choose a semisimple point 
$\tau\in H^\udot(X)$ and an admissible phase $\phi$. 
There exists a $\Z$-basis $\{[E_1],\dots,[E_N]\}$ of the 
$K$-group such that the matrix $(\chi(E_i,E_j))_{1\le i,j\le N}$ 
is uni-uppertriangular and that 
the asymptotic basis (see \S \ref{subsec:asymp_basis_GII}) 
of $X$ at $\tau$ with respect to $\phi$ is given by $\Gg_X \Ch(E_i)$, 
$i=1,\dots,N$. 
\end{theorem} 
\begin{proof} 
Recall from Remark \ref{rem:mutation} that the asymptotic basis 
makes sense also for the big quantum product $\star_\tau$ 
and that it changes by mutation as $\tau$ varies. 
By the theory of mutation, 
it suffices to prove the theorem at a single semisimple point $\tau$. 
If $\tau$ is in the image of the mirror map, we can take $[E_i]$ to be 
the mirror images of the Lefschetz thimbles (with phase $\phi$) 
under the isomorphism between the integral structures 
of the A-model and of the B-model, 
given in \cite[Theorem 4.11]{Iritani09}. 
\end{proof} 

\begin{remark}[\cite{Iritani09}]
Results similar to Theorems \ref{thm:toric_GammaI}, \ref{thm:toric_GammaII} 
hold for a weak-Fano toric orbifold. In the weak-Fano case, however, we 
need to take into consideration the effect of the mirror map. 
\end{remark} 
\begin{remark} 
Kontsevich's homological mirror symmetry suggests that the basis 
$[E_1],\dots,[E_N]$ in the $K$-group 
should be lifted to an exceptional collection in 
$D^b_{\rm coh}(X)$ because the corresponding 
Lefschetz thimbles form an exceptional collection in 
the Fukaya--Seidel category of the mirror. 
\end{remark} 

\section{Toric complete intersections} 
\label{sec:toric_CI} 
In this section we discuss Gamma conjecture I for a Fano complete 
intersection $Y$ in a toric manifold $X$. 
Again the problem comes down to the truth of a mirror analogue 
of Property $\O$. 

We begin with the remark that the number $T$ can be evaluated 
on the ambient part. 
Let 
\[
H_{\rm amb}^\udot(Y) := \im(i^* \colon  H^\udot(X)\to H^\udot(Y))
\]
denote the \emph{ambient part} of the cohomology group, 
where $i\colon Y \to X$ is the inclusion. 
It is shown in \cite[Proposition 4]{Pan98}, \cite[Corollary 2.5]{Iritani11} 
that the ambient part $H_{\rm amb}^\udot(Y)$ 
is closed under quantum multiplication $\star_\tau$ when 
$\tau \in H_{\rm amb}^\udot(Y)$. 
Note that Givental's mirror theorem \cite{Givental:toric} determines 
this ambient quantum cohomology $(H^\udot_{\rm amb}(Y), \star_\tau)$. 

\begin{proposition} 
\label{prop:ambient_evaluation}
The spectrum of $(c_1(Y)\star_0)$ on $H^\udot(Y)$ 
is the same as the spectrum of $(c_1(Y) \star_0)$ 
on $H^\udot_{\rm amb}(Y)$ as a set (when we ignore 
the multiplicities). 
In particular the number $T$ for $Y$ (see \eqref{eq:T}) 
can be evaluated on the ambient part.  
\end{proposition} 
\begin{proof} 
It follows from the fact that $c_1(Y)$ belongs to $H_{\rm amb}^\udot(Y)$ 
and the fact that $H^\udot(Y)$ is a module over the algebra 
$(H^\udot_{\rm amb}(Y),\star_0)$. 
\end{proof} 

We assume that $Y$ is a complete intersection in $X$ 
obtained from the following \emph{nef partition} 
\cite{Batyrev-Borisov, Givental:toric}. 
Let $D_1,\dots, D_m$ denote prime toric  divisors of $X$. 
Each prime toric divisor $D_i$ corresponds to a primitive generator 
$b_i\in \Z^n$ of a 1-dimensional cone of the fan of $X$. 
We assume that there exists a partition 
$\{1,\dots,m\} = I_0 \sqcup I_1 \sqcup \cdots \sqcup I_l$ 
such that $\sum_{i\in I_k} D_i$ is nef for $k=1,\dots,l$ 
and $\sum_{i\in I_0} D_i$ is ample. 
Let $\L_k := \O(\textstyle\sum_{i\in I_k} D_i)$ be the 
corresponding nef 
line bundle. We assume that $Y$ is the zero-locus of a transverse 
section of $\bigoplus_{k=1}^l \L_k$ over $X$. 
Then $c_1(Y) = c_1(X) - \sum_{k=1}^l c_1(\L_k) = 
\sum_{i\in I_0} [D_i]$ is ample and $Y$ is a Fano manifold. 
Define the Laurent polynomial functions $f_0,f_1,\dots, f_l \colon (\C^\times)^n 
\to \C$ as follows: 
\[
f_k(x)  =-  c_0 \delta_{0,k} + \sum_{i\in I_k} x^{b_i} 
\]
where $c_0$ is the constant arising from 
Givental's mirror map \cite{Givental:toric}: 
\[
c_0 := \sum_{\substack{d\in \Eff(X) : c_1(Y) \cdot d =1 \\ 
\qquad \quad [D_i]\cdot d\ge 0\, (\forall i)}} 
\frac{\prod_{k=1}^l (c_1(\L_k)\cdot d)! }
{\prod_{i=1}^m ([D_i] \cdot d)!}. 
\]
A mirror of $Y$ \cite{Hori-Vafa, Givental:toric} is the affine variety 
$Z := \{x \in (\C^\times)^n : f_1(x) = \cdots = f_l(x)=1\}$ 
equipped with a function $f_0 \colon Z \to \C$ and 
a holomorphic volume form 
\[
\omega_{Z}:= \frac{\frac{dx_1}{x_1}\wedge \cdots \wedge 
\frac{dx_n}{x_n}}
{d f_1 \wedge \cdots \wedge df_l}. 
\]
Consider the following condition, which implies Part (1) of Property $\O$ 
for $Y$. 

\begin{condition}[analogue of Property $\O$ for $Y$] 
\label{cond:compint_O}  
Let $Y$, $Z$ and $f_0 \colon Z \to \C$  
be as above. 
Let $Z_{\rm real} := 
\{(x_1,\dots,x_n) \in Z : x_i \in \R_{>0} \ (\forall i) \}$ 
be the positive real locus of $Z$. We have: 
\begin{itemize} 
\item[(a)] $Z$ is a smooth complete intersection of dimension $n-l$; 
\item[(b)] $f_0|_{Z_{\rm real}} \colon Z_{\rm real} \to \R$ 
attains global minimum at a unique critical point 
$x_{\rm con} \in Z_{\rm real}$; we call it \emph{the conifold point}; 
\item[(c)] the conifold point $x_{\rm con}$ is a non-degenerate 
critical point of $f_0$; 
\item[(d)] the number $T$ \eqref{eq:T} for $Y$ coincides 
with $T_{\rm con}:=f_0(x_{\rm con})$;  
\item[(e)] $T_{\rm con}$ is an eigenvalue of 
$(c_1(Y)\star_0)$ with multiplicity one. 
\end{itemize} 
\end{condition} 

\noindent 
{\bf Variant.}
We can consider the weaker version of Condition \ref{cond:compint_O} 
where part (e) is replaced with 
\begin{itemize}
\item[(e$'$)] $T_{\rm con}$ is an eigenvalue of $(c_1(Y)\star_0)
|_{H_{\rm amb}(Y)}$ 
with multiplicity one. 
\end{itemize} 
When (a), (b), (c), (d), (e$'$) hold, we say that $Y$ 
satisfies Condition \ref{cond:compint_O} \emph{on the ambient part}.

\begin{example} 
Let $Y$ be a degree $d$ hypersurface in $\P^n$ with $d<n+1$. 
The mirror is given by the function 
\[
f_0(x) = \frac{1}{x_1x_2\cdots x_n} + x_1 + \cdots + x_{n-d}
- \delta_{d,n} d!  
\]
on the affine variety $Z = \{ x \in (\C^\times)^n:
 f_1(x) = x_{n-d+1} + x_{n-d+2} + \cdots + x_n =1\}$. 
Following Przyjalkowski \cite{Przyjalkowski:comp}, 
introduce homogeneous co-ordinates 
$[y_{n-d+1}: \dots : y_n]$ of $\P^{n-d-1}$ and 
consider the change of variables: 
\[
x_i = \frac{y_i}{y_{n-d+1}+ \cdots + y_n}  \quad 
\text{for} \quad n-d+1 \le i\le n. 
\]
Then we have 
\[
f_0(x) = \frac{(y_{n-d+1} + \cdots + y_n)^d}{\prod_{i=1}^{n-d} x_i  
\cdot \prod_{j=n-d+1}^n y_j}
+ x_1 + \cdots + x_{n-d} - \delta_{n,d} d! 
\]
Setting $y_n=1$ we obtain a Laurent polynomial expression 
(with positive coefficients) for $f_0$ in the variables 
$x_1,\dots,x_{n-d}, y_{n-d+1},\dots,y_{n-1}$. 
Note that the change of variables maps $Z_{\rm real}$ 
isomorphically onto the real locus 
$\{x_i>0, y_j >0 : 1\le i\le n-d, n-d+1 \le j\le n-1\}$. 
As in the case of mirrors of toric manifolds (see \S\ref{sec:toric}), 
the Laurent polynomial expression of $f_0$ 
shows the existence of a conifold point in Condition \ref{cond:compint_O}. 
(Similarly, the mirrors of complete intersections in weighted projective spaces 
have Laurent polynomial expressions 
\cite{Przyjalkowski:comp}.) 

We have $T_{\rm con} = (n+1-d) d^{d/(n+1-d)} - \delta_{n,d} d!$.  
It is easy to check that $T_{\rm con}$ 
coincides with $T$ for $Y$ 
and gives a simple eigenvalue of 
$(c_1(Y)\star_0)|_{H^\udot_{\rm amb}(Y)}$ restricted to the ambient part 
(using Givental's mirror theorem \cite{Givental:toric}). 
Therefore $Y$ satisfies Condition \ref{cond:compint_O} on the ambient part. 

Furthermore, if the Fano index $n+1-d$ is greater than one, we have  
$c_1(Y)\star_0 \theta = 0$ for every primitive class $\theta \in H^{n-1}(Y)$. 
This follows from the fact that $(c_1(Y)\star_0)$ preserves the primitive 
part in $H^{n-1}(Y)$ and for degree reasons. 
Also if $\dim Y = n-1$ is odd, $Y$ has no even degree primitive classes 
(see Appendix \ref{app:odd} for the treatment of odd classes). 
Therefore we conclude that $Y$ satisfies 
Condition \ref{cond:compint_O} except possibly for 
even-dimensional index-one hypersurfaces.  
\end{example} 

\begin{remark}[\cite{Iritani11}]
\label{rem:ineq}
Using an inequality of the form $\beta_1 u_1 + \cdots + \beta_m u_m 
\ge u_1^{\beta_1} \cdots u_m^{\beta_m}$ for $u_i>0$, 
$\beta_i>0$, $\sum_i \beta_i =1$, we find that 
$f_0|_{Z_{\rm real}}$ is bounded (from below) by a convex function: 
\[
f_0(x) + l = f_0(x) + f_1(x) + \cdots + f_l(x)  \ge 
-c_0 + 
\epsilon \sum_{i=1}^n (x_i^{1/k} + x_i^{-1/k}) 
\qquad 
\forall x \in Z_{\rm real} 
\]
where $\epsilon>0$ and $k>0$ are constants. 
In particular $f_0|_{Z_{\rm real}} \colon Z_{\rm real} \to \R$ 
is proper and attains global minimum at some point. 
It also follows that the oscillatory integral \eqref{eq:osciint_compint} below 
converges. 
\end{remark} 

\begin{theorem}
\label{thm:toriccomp_GammaI}
Let $Y$ be a Fano complete intersection in a toric 
manifold constructed from a nef partition as above. 
If $Y$ satisfies Condition \ref{cond:compint_O}, 
$Y$ satisfies Gamma Conjecture I. 
If $Y$ satisfies Condition \ref{cond:compint_O} on the 
ambient part, the ambient quantum cohomology of $Y$ satisfies 
Gamma conjecture I. 
\end{theorem} 
\begin{proof} 
In \cite[Theorem 5.7]{Iritani11}, it is proved that the quantum 
cohomology central charge (see Remark \ref{rem:centralcharges}) 
of $\O_Y$ has an integral representation: 
\begin{equation}
\label{eq:osciint_compint}
Z(\O_Y) = z^{\frac{\dim Y}{2}} 
\left(1, S(z)z^{-\mu}z^{c_1(Y)}\Gg_Y \right)_X  
= \int_{Z_{\rm real}} e^{-f_0/z} \omega_Z. 
\end{equation} 
We note that the constant term $c_0$ in $f_0$ comes from the value of 
the mirror map at $\tau=0$. 
In \cite{Iritani11}, a slightly more general statement was proved: 
we have mirrors $(Z_\tau,f_{0,\tau},\omega_{Z,\tau})$ 
parameterized by $\tau \in H^2_{\rm amb}(Y)$ 
and 
\[
z^{\frac{\dim Y}{2}} \left(1, S(\tau,z) z^{-\mu} z^{c_1(Y)} \Gg_Y 
\right)_Y = \int_{Z_{\tau,{\rm real}}} e^{f_{0,\tau}/z} \omega_{Z,\tau}
\]
where $S(\tau,z)z^{-\mu} z^{c_1(Y)}$ 
is the fundamental solution for the big quantum 
connection \eqref{eq:big_qconn} which restricts to 
the one in Proposition \ref{prop:fundsol} at $\tau=0$.  
By differentiating this in the $\tau$-direction, we obtain 
\[
z^{\frac{\dim Y}{2}} \left( z \nabla_{\alpha_1} 
\cdots z \nabla_{\alpha_k} 1, S(\tau,z) z^{-\mu} z^{c_1(Y)} \Gg_Y 
\right)_Y = z \partial_{\alpha_1} \cdots z \partial_{\alpha_k} 
\int_{Z_{\tau, {\rm real}}} e^{-f_{0,\tau}/z} \omega_{Z,\tau} 
\]
for $\alpha_1,\dots,\alpha_k \in H^2(Y)$. 
Since $H^\udot_{\rm amb}(Y)$ is generated by $H^2_{\rm amb}(Y)$, 
we obtain an integral representation of the form: 
\[
z^{\frac{\dim Y}{2}} \left(\phi, S(z) z^{-\mu}z^{c_1(Y)}
\Gg_Y \right)_Y 
= \int_{Z_{\rm real}} \varphi(x,z) e^{-f_0/z} \omega_Z 
\]
for every $\phi \in H^\udot_{\rm amb}(Y)$ (for some 
function $\varphi(x,z)$). 
Since the flat section $S(z) z^{-\mu} z^{c_1(Y)} \Gg_Y$ 
takes values in the ambient part $H_{\rm amb}^\udot(Y)$, 
we obtain integral representations for all components 
of $S(z) z^{-\mu} z^{c_1(Y)} \Gg_Y$. 
These integral representations and Condition \ref{cond:compint_O} 
show that $\|e^{T/z}S(z) z^{-\mu} z^{c_1(Y)} \Gg_Y\|$ grows at most 
polynomially as $z\to +0$. 
\end{proof}

\section{Quantum Lefschetz} 
\label{sec:Lefschetz} 
In this section we show that Gamma conjecture I 
is compatible with the quantum Lefschetz principle. 

Let $X$ be a Fano manifold of index $r\ge 2$. 
We write $-K_X = rh$ for an ample class $h$. 
Let $Y\subset X$ be a degree-$a$ Fano hypersurface 
in the linear system $|ah|$ with $0<a<r$. 
Assuming the truth of Gamma Conjecture I for $X$, we study 
Gamma Conjecture I for $Y$. 

We write the $J$-function of $X$ \eqref{eq:J-function} as 
\[
J_X(t) = e^{rh \log t} \sum_{n=0}^\infty J_{rn} t^{rn} 
\]
with $J_d\in H^\udot(X)$. By the quantum Lefschetz theorem 
\cite{YPLee:qLefschetz, Coates-Givental}, 
the $J$-function of $Y$ is 
\begin{equation}
\label{eq:qLefschetz_J} 
J_Y(t) = e^{(r-a)h \log t-c_0 t} \sum_{n=0}^\infty (ah+1) \cdots (ah+an) 
(i^*J_{rn})  t^{(r-a)n}, 
\end{equation} 
where $i\colon Y \to X$ is the natural inclusion and $c_0$ is the constant: 
\begin{equation} 
\label{eq:c0_general_hypersurface} 
c_0 = \begin{cases} 
a! \Ang{[\pt],J_r} = 
a! \sum_{h \cdot d = 1} 
\langle [\pt] \psi^{r-2} \rangle_{0,1,d}^X & \text{if $r-a=1$}; \\ 
0 & \text{if $r-a>1$}. 
\end{cases} 
\end{equation} 
The quantum Lefschetz principle can be rephrased 
in terms of the Laplace transformation: 
\begin{lemma} 
\label{lem:Laplace}
We have 
\[ 
J_Y(t=u^{a/(r-a)}) = \frac{e^{-c_0t}}{\Gamma(1+ah) u} 
\int_0^\infty  i^* J_X(q^{a/r}) e^{-q/u} dq. 
\]
\end{lemma} 
\begin{remark} 
The Laplace transformation in the above lemma converges 
for sufficiently small $u>0$ 
because of the exponential asymptotics as $t\to +\infty$ 
in Proposition \ref{prop:J_asymp} and the growth estimate 
$\|J_X(t)\| \le C |\log t|^{\dim X}$ as $t\to +0$. 
\end{remark} 

Suppose that $X$ satisfies Gamma Conjecture $I$. 
Recall from Proposition \ref{prop:J_asymp} that we have 
the following limit formula for $J_X$: 
\begin{equation} 
\label{eq:J_X_asymp}
\Gg_X \propto \lim_{t\to +\infty} t^{\frac{\dim X}{2}} e^{-T_X t} J_X(t). 
\end{equation} 
Here  $T_X$ is the number $T$ in \eqref{eq:T} for $X$. 
Using the stationary phase approximation in Lemma \ref{lem:Laplace}, 
we obtain the following result. 

\begin{theorem} 
Suppose that the $J$-function of $X$ satisfies 
the limit formula \eqref{eq:J_X_asymp}
and let $Y$ be a Fano hypersurface in the linear system $\lvert-(a/r) K_X\rvert$, 
where $r$ is the Fano index of $X$ and $0<a<r$. 
Then the $J$-function of $Y$ satisfies the limit formula: 
\[
\Gg_Y \propto \lim_{t\to + \infty} 
t^{\frac{\dim Y}{2}}e^{-(T_0 - c_0)t}  J_Y(t) 
\]
where $c_0$ is given in \eqref{eq:c0_general_hypersurface} 
and the positive number $T_0>0$ is determined by the relation: 
\[
\left(\frac{T_0}{r-a}\right)^{r-a} = a^a \left( \frac{T_X}{r}\right)^r. 
\]
In particular, if $Y$ satisfies Property $\O$ (Definition \ref{def:propertyO}), 
then $Y$ satisfies Gamma conjecture I by Corollary \ref{cor:limitGamma}. 
\end{theorem} 
\begin{proof} 
We write $n = \dim X$ and $T= T_X$ for simplicity. 
We set $\tJ(t) :=  t^{\frac{n}{2}}e^{-T t} i^*J_X(t)$. 
The limit formula \eqref{eq:J_X_asymp} 
gives $\lim_{t\to +\infty} \tJ(t) = C_1 i^* \Gg_X$ for some $C_1 \neq 0$. 
By Lemma \ref{lem:Laplace}, we have 
\begin{align*} 
t^{\frac{n-1}{2}}e^{-(T_0-c_0)t}J_Y(t = u^{a/(r-a)}) & = 
\frac{t^{\frac{n-1}{2}}e^{-T_0 t}}{\Gamma(1+ah) u} 
\int_0^\infty  q^{- an/(2r)} e^{T q^{a/r}-q/u} \tJ(q^{a/r}) 
dq \\
& = 
\frac{\sqrt{t}}{\Gamma(1+ah)} 
\int_0^\infty q^{-an/(2r)} 
e^{-(q-T q^{a/r}+T_0)t} \tJ(t q^{a/r})  dq 
\end{align*} 
where in the second line we performed the change of 
variables $q\to u^{r/(r-a)}q$ and used $t = u^{a/(r-a)}$. 
Consider the function $\theta(q)= q-Tq^{a/r}+T_0$ on 
$[0,\infty)$. This function has a unique critical point at 
$q_0 := (\frac{a}{r} T)^{\frac{r}{r-a}}$ and 
attains a global minimum at $q=q_0$; we have $\theta(q_0) = 0$ 
by the definition of $T_0$. 
By the stationary phase approximation, the $t\to +\infty$ 
asymptotics of this integral is determined by the 
behaviour of the integrand near $q=q_0$. 
To establish the asymptotics rigorously, we divide the interval $[0,\infty)$ 
of integration 
into $[0,q_0/2]$ and $[q_0/2,\infty)$. 
We first estimate the integral over $[0,q_0/2]$. 
\begin{multline}
\label{eq:the_first_half}
\left \| 
\sqrt{t} \int_0^{q_0/2} 
q^{-an/(2r)} 
e^{-\theta(q) t} \tJ(t q^{a/r})  dq 
\right \|
\le  e^{- t \theta(q_0/2)} \sqrt{t} 
\int_0^{q_0/2} q^{-an/(2r)} 
\|\tJ(tq^{a/r})\| dq \\ 
\le e^{- t \theta(q_0/2)} t^{\frac{n+1}{2}- \frac{r}{a}} 
\frac{r}{a} 
\int_0^{t (q_0/2)^{a/r}} 
y^{-n/2} \|\tJ(y)\| y^{(r/a)-1} dy.   
\end{multline} 
where we set $y = t q^{a/r}$. 
Note that $y^{-n/2}\|\tJ(y)\| y^{(r/a)-1}$ is integrable 
near $y=0$ (by the definition of $\tJ$) 
and is of polynomial growth as $y\to +\infty$. 
Therefore the integral 
\[
\int_0^{t (q_0/2)^{a/r}} 
y^{-n/2} \|\tJ(y)\| y^{(r/a)-1} dy.   
\]
is of polynomial growth 
as $t \to +\infty$. Since $\theta(q_0/2)>0$, the integral 
\eqref{eq:the_first_half} 
over $[0,q_0/2]$ goes to zero (exponentially) as $t\to +\infty$. 
Next we consider the integral over $[q_0/2,\infty)$. 
By the change of variables $q=q_0 e^{x/\sqrt{t}}$, 
the integral over $[q_2/2,\infty)$ can be written as 
\begin{equation} 
\label{eq:apply_DCT}
\frac{q_0^{1-\frac{an}{2r}}}{\Gamma(1+ah)} 
\int_{-\sqrt{t} \log 2}^\infty 
e^{(1-\frac{an}{2r}) x/\sqrt{t}} 
e^{- t\theta(q_0 e^{x/\sqrt{t}})} 
\tJ\left(t q_0^{a/r} e^{\frac{a}{r} x/\sqrt{t}} \right) dx 
\end{equation} 
Since we have an expansion of the form 
$\theta(q_0 e^{y}) = a_2 y^2 + \sum_{n=3}^\infty
a_n y^n$, we have for a fixed $x\in \R$,  
\[
e^{(1-\frac{an}{2r})x/\sqrt{t}} \to 1, \qquad 
e^{-t\theta(q_0 e^{x/\sqrt{t}})} \to e^{-a_2 x^2}, 
\qquad 
\tJ\left(t q_0^{a/r} e^{\frac{a}{r} x/\sqrt{t}} \right) 
\to C_1 i^*\Gg_X 
\]
as $t\to +\infty$. On the other hand, since $\theta(q_0 e^y)$ grows 
exponentially as $y\to +\infty$, we have an estimate 
of the form 
\[
\theta(q_0 e^y) \ge C_2 y^2 
\]
on $y\in [-\log 2, \infty)$ for some $C_2>0$. 
Therefore, when $t\ge 1$, the integrand of \eqref{eq:apply_DCT} 
can be estimated as 
\[
\left\|e^{(1-\frac{an}{2r}) x/\sqrt{t}} 
e^{- t\theta(q_0 e^{x/\sqrt{t}})} 
\tJ\left(t q_0^{a/r} e^{\frac{a}{r} x/\sqrt{t}} \right) 
\right\| 
\le C_3 e^{|1-\frac{an}{2r}| |x|} e^{-C_2 x^2} 
\]
for all $x\in [-\sqrt{t} \log 2, \infty)$ for some $C_3>0$. 
Thus the integrand is uniformly bounded by an integrable function 
and we can apply Lebesgue's dominated convergence theorem to 
\eqref{eq:apply_DCT} to see that the limit of \eqref{eq:apply_DCT} 
as $t\to +\infty$ is proportional to $i^*\Gg_X/\Gamma(1+ah)$. 
The conclusion follows by noting that 
\[
\Gg_Y = \frac{i^*\Gg_X}{\Gamma(1+ah)}. 
\]
\end{proof} 

It is natural to ask the following questions: 
\begin{problem} 
Check that $T_0-c_0$ is the number $T$ \eqref{eq:T} for the hypersurface $Y$. 
Also prove that $Y$ satisfies Property $\O$ (assuming that $X$ 
satisfies Property $\O$). 
\end{problem} 

\begin{remark}[\cite{Galkin:Apery}] 
It is easy to see that 
Ap\'ery limits \eqref{eq:Apery_limit} (or Gamma conjecture I') are compatible 
with quantum Lefschetz. 
Suppose that $X$ is a Fano manifold of index $r$ satisfying Gamma conjecture I' 
and let $Y \in \lvert-(a/r)K_X\rvert$ be a Fano hypersurface 
of index $r-a>1$. 
Then for any $\alpha \in H_\ldot(Y)$ with $c_1(Y) \cap \alpha =0$, 
we have the limit formula: 
\[
\lim_{n\to \infty} \frac{\Ang{i_*\alpha, J_{X,rn}}}
{\Ang{[\pt],J_{X,rn}}}
 = \langle i_*\alpha,\Gg_X  \rangle  = \langle \alpha, \Gg_Y\rangle 
\]
where $J_{X,n}$ is the Taylor coefficients of the $J$-function of $X$ 
(as in \eqref{eq:J_expansion}) and we used $h\cap i_*\alpha =0$ 
in the second equality. Since the index $r-a$ is greater than one, 
the quantum Lefschetz \eqref{eq:qLefschetz_J} gives:  
\[
\Ang{\alpha, J_{Y,(r-a)n}} = (an)! \Ang{i_*\alpha, J_{X,rn}}, 
\quad 
\Ang{[\pt], J_{Y,(r-a)n}} = (an)! \Ang{[\pt],J_{X,rn}} 
\]
and thus $Y$ also satisfies the Gamma conjecture I'.  
\end{remark} 

\section{Grassmannians} 
\label{sec:Grassmannian}
In this section, we prove Gamma conjecture I' 
(Conjecture \ref{conj:I_dash}) for Grassmannians  
$\Gr(r,n)$ using the Hori--Vafa mirror \cite{Hori-Vafa} and abelian/non-abelian 
correspondence \cite{BCFK}. 
The discussion in this section extends the proof of Dubrovin conjecture 
for Grassmannians by Ueda \cite{Ue05}. 
In the course of the proof, we obtain a formula for quantum cohomology 
central charges of $\Gr(r,n)$ in terms of mirror oscillatory integrals. 
We note that Gamma conjecture I' for $\Gr(2,n)$ was proved 
by Golyshev \cite{Golyshev08a}. 
Gamma conjectures I and II were proved 
for general $\Gr(r,n)$ in \cite{GGI:gammagrass} by a different 
(but closely related) method 
based on the quantum Satake principle \cite{GMa}. 

\subsection{Abelian quotient and non-abelian quotient}
Let $\G=\Gr(r,n)$ denote the Grassmann variety of $r$-dimensional 
subspaces in $\C^n$ 
and let $\P = \P^{n-1} \times \cdots \times \P^{n-1}$ ($r$ times) 
denote the product of $r$ copies of the projective space $\P^{n-1}$. 
We relate these two spaces in the framework of 
Martin \cite{Martin00}: 
$\G$ and $\P$ arise as 
non-abelian and abelian quotients of the same vector space 
$\Hom(\C^r,\C^n)$: 
\[
\G = \Hom(\C^r,\C^n) /\!/ GL(r), \quad 
\P = \Hom(\C^r,\C^n) /\!/ (\C^\times)^r. 
\]  
We have a rational map 
\[
\P \dasharrow \G 
\]
sending a collection of lines $l_1,\dots,l_r$ to 
the subspace spanned by $l_1,\dots,l_r$. 
We can also relate $\P$ and $\G$ by the following diagram 
\cite{Martin00}: 
\begin{equation}
\label{eq:diagPG} 
\begin{CD} 
\F @>{p}>> \G \\ 
@V{\iota}VV   @. \\ 
\P 
\end{CD}
\end{equation} 
where $\F : = \operatorname{Fl}(1,2,\dots,r,n)$ is the 
partial flag variety, $p \colon \F \to \G$ is the natural projection 
and $\iota$ is a real-analytic embedding which sends  
a flag $0\subset V_1 \subset V_2 \subset 
\cdots \subset V_r \subset \C^n$ to the lines 
$L_1,\dots,L_r$ such that $L_i$ is the orthogonal 
complement of $V_{i-1}$ in $V_{i}$. 
Here we need to choose a Hermitian metric on $\C^n$ 
for the definition of $\iota$. 
The diagram \eqref{eq:diagPG} naturally 
comes from the following description 
of $\G$ and $\P$ as symplectic reductions 
of $\Hom(\C^r,\C^n)$: 
\begin{equation} 
\label{eq:symp_reduction}
\G \cong \mu_G^{-1}(\eta) / G, \quad 
\F \cong \mu_G^{-1}(\eta) / T, \quad 
\P \cong \mu_T^{-1}(\eta_0) /T.   
\end{equation}
Here $\mu_G$ and $\mu_T$ are the moment maps 
of $G=U(r)$ and $T=(S^1)^r$-actions on 
$\Hom(\C^r,\C^n)$ respectively 
and $\eta\in \Lie(G)^\star$ and $\eta_0\in \Lie(T)^\star$ 
are non-zero central elements (i.e.~scalar multiples 
of the identify matrix) such that $\pi(\eta) = \eta_0$ 
(where $\pi \colon \Lie(G)^\star \to \Lie(T)^\star$ is the natural 
projection): 
\begin{align*} 
\mu_G (A) = A^\dagger A, \quad 
\mu_T (A) = \left((A^\dagger A)_{i,i}\right)_{i=1}^r, \quad 
\xymatrix{
\Hom(\C^r,\C^n) \ar[r]^{\phantom{AB}\mu_G} \ar[rd]_{\mu_T} & 
\Lie(G)^\star \ar[d]^{\pi} \\ 
& \Lie(T)^\star 
}
\end{align*} 
The map $\iota$ in \eqref{eq:diagPG} is then induced from 
the inclusion $\mu_G^{-1}(\eta) \subset \mu_T^{-1}(\eta_0)$. 

We describe the cohomology groups of $\G$ and $\P$. 
Let $\V_1\subset \V_2\subset \cdots \subset \V_r$ 
denote the tautological bundles (with $\rank \V_i = i$) over $\F$ 
and set 
\[
\L_i := (\V_i / \V_{i-1})^\vee.  
\]
The line bundle $\L_i$ is the pull-back of the line bundle 
$\O(1)$ over the $i$th factor $\cong \P^{n-1}$ of $\P$
under the non-algebraic map $\iota$.  
We denote the corresponding line bundle over $\P$ 
also by $\L_i$, i.e. $\iota^\star \L_i = \L_i$.  
We set 
\[
x_i := c_1(\L_i) \in \text{$H^\udot(\P)$ or $H^\udot(\F)$}. 
\] 
A symmetric polynomial in $x_1,\dots,x_r$ can be written 
in terms of Chern classes of $\V_r$ and thus makes sense 
as a cohomology class on $\G$. 
The cohomology rings of $\P$, $\G$ and $\F$ are 
described in terms of $x_i$ as follows: 
\begin{align*} 
H^\udot(\P) & \cong \C[x_1,\dots,x_r] 
\Big/\langle x_1^{n},\dots, x_r^n \rangle, \\
H^\udot(\G) & \cong \C[x_1,\dots,x_r]^{\frS_r}
\Big/ \langle h_{n-r+1},\dots,h_{n} \rangle, \\ 
H^\udot(\F) & \cong \C[x_1,\dots,x_r] 
\Big/ \langle h_{n-r+1},\dots, h_n \rangle. 
\end{align*} 
Here $h_i=c_i(\C^n/\V_r)$ is the $i$th complete symmetric 
function of $x_1,\dots,x_r$. 
An additive basis of $H^\udot(\G)$ is given by 
the Schubert classes: 
\begin{equation} 
\label{eq:Schubert_class}
\sigma_\mu := \frac{\det\left(x_i^{\mu_j+r-j}\right)_{1\le i,j\le r}}
{\det\left(x_i^{r-j}\right)_{1\le i,j\le r}}
\end{equation} 
where $\mu$ ranges over partitions such that 
$n-r \ge \mu_1 \ge \mu_2 \ge \cdots \ge 
\mu_r\ge 0$. 
This is the Poincar\'e dual of the Schubert cycle 
\[
\Omega_\mu = \{ V \in \Gr(r,n) : \dim( V \cap \C^{n-r+i-\mu_i}) \ge i, 
\ 1\le i\le r\} 
\]
and $\deg \sigma_\mu = 2 |\mu| = 2 \sum_{i=1}^r \mu_i$ 
(degree as a cohomology class), see 
\cite[\S 6, Chapter 1]{Griffiths-Harris}. 

Let $J_{\P^{n-1}}(t)$ be the $J$-function 
(see \eqref{eq:J-function}, \eqref{eq:J-function_P}) 
of $\P^{n-1}$. 
We define the multivariable $J$-function of $\P$ by 
\[
J_\P(t_1,\dots,t_r) = J_{\P^{n-1}}(t_1) \otimes 
J_{\P^{n-1}}(t_2) \otimes \cdots \otimes J_{\P^{n-1}}(t_r).  
\]
This takes values in $H^\udot(\P) \cong H^\udot(\P^{n-1})^{\otimes r}$. 
Bertram--Ciocan-Fontanine--Kim \cite{BCFK} proved the 
following \emph{abelian/non-abelian correspondence} 
between the $J$-functions of $\G$ and $\P$. 
\begin{theorem}[\cite{BCFK}]
\label{thm:BCFK}
The $J$-function $J_\G(t)$ of $\G$ is given by 
\[
p^\star J_\G(t) 
= e^{- \sigma_1 \pi \iu (r-1) } 
\left[\frac{\left( 
\prod_{1\le i<j\le r} (\theta_i - \theta_j) \right) 
\iota^\star J_{\P}(t_1,\dots,t_r)}{n^{\binom{n}{r}}\Delta} 
\right]_{t_1 = \cdots = t_r = \xi t} 
\] 
where $\Delta := \prod_{i<j} (x_i - x_j)$, 
$\theta_i := t_i \parfrac{}{t_i}$, $\xi := e^{\pi\iu(r-1)/n}$ and 
$\sigma_1 := x_1 + \dots + x_r$.  
\end{theorem} 

\subsection{Preliminary lemmas} 
We discuss elementary topological properties of the maps 
in \eqref{eq:diagPG}. Martin \cite{Martin00} has shown 
similar results for general abelian/non-abelian quotients.   

\begin{lemma}
\label{lem:tangentbundles} 
Let $\N_\iota \to \F$ denote the normal bundle of $\iota$ 
and let $\T_p = \Ker(dp) \subset T\F$ denote the relative tangent 
bundle of $p$. Then the complex structure $I$ on $\P$  
induces an isomorphism: 
\[
I \colon \T_p \cong \N_\iota. 
\]
In particular, we have an isomorphism 
\[
\iota^\star T \P  \cong p^\star T\G \oplus 
\left(\T_p \otimes_\R \C, \iu\right)  
\]
of topological complex vector bundles, 
where the complex structure $\iu$ on $\T_p \otimes_\R \C$ 
is induced from that on the $\C$-factor. 
\end{lemma} 
\begin{proof} 
Recall that $\P$, $\G$, $\F$ are described as symplectic 
reductions \eqref{eq:symp_reduction}. 
Note that $\rank \N_\iota = \rank \T_p = \dim G -\dim T$. 
Let $\underline{X}$ denote the fundamental vector field 
associated with $X\in \Lie(G)$. 
For $x\in \mu_G^{-1}(\eta)$, it follows from the 
property of the moment map that $I \underline{X}_x$ is 
perpendicular to $T_x \mu_G^{-1}(\eta)$ for $X\in \Lie(G)$. 
Thus we have an orthogonal decomposition: 
\[
T_x \Hom(\C^r,\C^n) = T_x\mu_G^{-1}(\eta) \oplus I T_x(G\cdot x). 
\]
Similarly we have an orthogonal decomposition 
\[
T_x \Hom(\C^r,\C^n) = T_x\mu_T^{-1}(\eta_0) \oplus I T_x(T \cdot x).  
\]
Note that the tangent space $T_{[x]} \P$ at $[x]\in \P=\mu_T^{-1}(\mu_0)/T$ 
is identified with the orthogonal complement 
of $T_x(T\cdot x)$ in $T_x \mu_T^{-1}(\eta_0)$, 
which is preserved by the complex structure $I$. 
Under this identification, the orthogonal complement of $T_x(T\cdot x)$ 
in $T_x (G\cdot x)$ is identified with the fiber $(\T_p)_{[x]}$; 
by the above decompositions, the complex structure $I$ sends 
this to the orthogonal complement of $T_x \mu_G^{-1}(\eta)$ in 
$T_x \mu_T^{-1}(\eta_0)$. This shows 
$I((\T_p)_{[x]}) = (\N_{\iota})_{[x]}$ and the conclusion follows. 
\end{proof} 

\begin{remark} 
\label{rem:complexification} 
Note that as complex vector bundles   
\[
(\T_p \otimes_\R \C, \iu)  \cong \T_p \oplus \ov\T_p 
\]
where $\ov\T_p$ denotes the conjugate of the complex 
vector bundle $\T_p$. 
\end{remark} 

Because $\F, \P$ are compact oriented manifolds,  
the push-forward map $\iota_* \colon H^\udot(\F) \to H^\udot(\P)$ 
is well-defined (even though $\iota$ is not algebraic). 
Here we need to be a little careful about the orientation of 
the normal bundle of $\iota$. 

\begin{lemma}
\label{lem:iota_star}
The push-forward map $\iota_\star \colon H^\udot(\F) \to H^\udot(\P)$ 
is given by 
\[
\iota_\star(f(x)) = f(x) \cup \ov{\Delta}   
\]
for any polynomial $f(x) = f(x_1,\dots,x_r)$. 
Here $\ov{\Delta} = \prod_{i<j} (x_j-x_i) = 
(-1)^{\binom{r}{2}} \Delta$. 
\end{lemma} 
\begin{proof}
The normal bundle $\N_\iota$ of $\iota$ is not a complex vector bundle, 
but can be written as the formal difference of complex 
vector bundles: 
\[
\N_\iota = \iota^\star T\P \ominus T\F 
\]
and thus defines an element of the $K$-group $K^0(\P)$ 
of topological complex vector bundles. 
Using Lemma \ref{lem:tangentbundles}, Remark \ref{rem:complexification} 
and a topological isomorphism $T\F \cong \T_p \oplus p^\star T\G$, 
we have: 
\[
\N_\iota = \T_p \otimes \C \ominus \T_p 
=  \ov\T_p.   
\]
On the other hand, we have a topological isomorphism: 
\[
\T_p \cong \bigoplus_{i<j} \Hom(\L_i^\vee,\L_j^\vee).   
\]
Thus $\ov\T_p\to \F$ is isomorphic to the restriction 
of the following complex vector bundle over $\P$:  
\[
\bigoplus_{i<j} \ov{\L_i} \otimes \L_j \longrightarrow \P 
\]
and $\F$ is the zero locus of the section of this bundle  
defined by the given Hermitian metric on $\C^n$. 
Therefore, $\iota_\star (f(x)) = f(x) \cup \ov{\Delta}$. 
\end{proof} 

\begin{lemma}
\label{lem:p_star} 
For a polynomial $f(x) = f(x_1,\dots, x_r)$, we have 
\[
p_\star(f(x)) = \frac{1}{\Delta} 
\left ( \sum_{\sigma \in \frS_r} \sgn(\sigma) 
f(x_{\sigma(1)}, \dots, x_{\sigma(r)}) 
\right). 
\]
Note that the right-hand side is a symmetric polynomial 
in $x_1,\dots,x_r$. 
\end{lemma} 
\begin{proof}
The $\frS_r$-action on $\P$ induces 
a non-algebraic $\frS_r$ action on $\F$ 
which preserves the fibration 
$p \colon \F \to \G$. 
The action of $\sigma \in \frS_r$ on $\F$ 
changes the orientation of each fiber of $p$  
by $\sgn(\sigma)$. 
Therefore $p_\star(f(x)) = \sgn(\sigma) p_\star 
(f(x_{\sigma(1)},\dots,x_{\sigma(r)}))$. 
Thus we have 
\begin{align*} 
p_\star(f(x)) & = \frac{1}{r!} 
p_\star\left( \sum_{\sigma \in \frS_r} \sgn(\sigma) 
f(x_{\sigma(1)},\dots,x_{\sigma(r)}) \right)  \\ 
& = \frac{1}{r!} 
\frac{\sum_{\sigma \in \frS_r} \sgn(\sigma) 
f(x_{\sigma(1)},\dots,x_{\sigma(r)})}{\Delta} 
p_\star(\Delta).  
\end{align*} 
The conclusion follows from $p_\star(\Delta) = 
p_\star(\Euler(\T_p)) = r!$. 
\end{proof} 

\subsection{Comparison of cohomology and $K$-groups}
Using the diagram \eqref{eq:diagPG}, we identify 
the cohomology (or the $K$-group) of $\G$ with the 
anti-symmetric part of the cohomology (resp.~the $K$-group) of $\P$. 
This identification was proved by Martin \cite{Martin00} for 
cohomology groups of general abelian/non-abelian quotients 
and has been generalized to $K$-groups by 
Harada--Landweber \cite{Harada-Landweber}. 
We compare the cohomological and $K$-theoretic 
identifications via the map 
$\E \mapsto \Gg_F \Ch(\E)$. 

We identify the $r$th wedge product 
$\wedge^r H^\udot(\P^{n-1})$ with 
the anti-symmetric part of $H^\udot(\P)$ 
(with respect to the $\frS_r$-action) via the map: 
\[
\wedge^r H^\udot(\P^{n-1}) \ni 
\alpha_1 \wedge \alpha_2 \wedge \cdots \wedge \alpha_r 
\longmapsto 
\sum_{\sigma\in \frS_r} \sgn(\sigma) 
\alpha_{\sigma(1)} \otimes \alpha_{\sigma(2)} 
\otimes \cdots \otimes \alpha_{\sigma(r)} 
\in H^\udot(\P). 
\]
Similarly we identify $\wedge^r K^0(\P)$ with the anti-symmetric 
part of $K^0(\P)$, where $K^0(\cdot)$ denotes the topological $K$-group. 
For a non-increasing sequence $\mu_1\ge \mu_2\ge \cdots \ge \mu_r$ 
of integers, we set 
\[
\E_\mu := p_\star\left(\L_1^{\mu_1} \otimes \L_2^{\mu_2} 
\otimes \cdots \otimes \L_r^{\mu_r} \right). 
\]
By the Borel--Weil theory, $\E_\mu$ is the vector bundle on $\G$ associated 
to the irreducible $GL(r)$-representation of highest weight $\mu$. 
Shifting $\mu_i$ by some number $\mu_i \mapsto \mu_i + \ell$ 
simultaneously corresponds to twisting the $GL(r)$-representation 
by $\det^{\otimes \ell}$, where $\det \colon GL(r) \to \C^\times$ 
corresponds to $\O(1)$ on $\G$. 
These vector bundles span the $K$-group of $\G$. 
We define a line bundle $\cR$ on $\F$ by 
\[
\cR = \L_1^{-(r-1)} \otimes \L_2^{-(r-2)} \otimes 
\cdots \otimes \L_{r-1}^{-1}. 
\]
This restricts to a half-canonical bundle\footnote{corresponding to the 
half-sum $\rho$ of positive roots} on each fiber of $p$. 
One can easily check that: 
\begin{equation} 
\label{eq:half-canonical}
\det(\T_p^\vee) = p^\star(\O(r-1)) \otimes \cR^{\otimes 2} 
\end{equation} 
where $p^\star \O(1) = \L_1\L_2\cdots \L_r$. 
We have the following result: 
\begin{proposition} 
\label{prop:comparison_coh_K}
Let $\mu_1\ge \mu_2 \ge \cdots \ge \mu_r$ be a non-increasing 
sequence  
and let $k_1>k_2>\cdots > k_r$ be the strictly 
decreasing sequence given by $k_i = \mu_i +r-i$. 
Then we have: 
\begin{enumerate}
\item The map $(r!)^{-1} p_\star \iota^\star \colon 
\wedge^r H^\udot(\P) \to H^\udot(\G)$ sends 
the class $x_1^{k_1} \wedge x_2^{k_2} \wedge \cdots \wedge x_r^{k_r}$ 
to the Schubert class $\sigma_\mu$ \eqref{eq:Schubert_class} 
and is an isomorphism. 

\item The map $(r!)^{-1} p_\star (\cR \otimes \iota^\star(\cdot)) 
\colon \wedge^r K^0(\P^{n-1}) \to K^0(\G)$ sends 
the class $\L_1^{k_1} \wedge \L_2^{k_2} \wedge \cdots \wedge \L_r^{k_r}$ 
to the class $\E_\mu$ and is an isomorphism of vector spaces preserving 
the natural $\Z$-lattices. 

\item We have the commutative diagram:
\begin{equation} 
\label{eq:comparison_coh_K}
\begin{aligned}
\xymatrix{
\wedge^r K^0(\P^{n-1}) 
\ar[rrrr]^{\frac{1}{r!}p_\star(\cR \otimes \iota^\star(\cdot))} 
\ar[d]^{\Gg_\P \Ch(\cdot)} 
& & & &
K^0(\G)  
\ar[d]^{\Gg_\G \Ch(\cdot)} 
\\ 
\wedge^r H^\udot(\P^{n-1}) 
\ar[rrrr]^{ \phantom{AB} \frac{\epsilon_r}{r!} e^{-\pi\iu(r-1)\sigma_1} 
 p_\star \iota^\star}
& & & & H^\udot(\G) 
}
\end{aligned} 
\end{equation}
where $\epsilon_r = (2\pi\iu)^{-\binom{r}{2}}$ 
and $\Ch(E) = \sum_{p\ge 0} (2\pi\iu)^p \ch_p(E)$ is the 
modified Chern character.  
\end{enumerate} 
\end{proposition} 
\begin{proof} 
(1): This follows from Lemma \ref{lem:p_star} and 
the definition of the Schubert class $\sigma_\mu$ 
\eqref{eq:Schubert_class}. 

(2): Using the Grothendieck--Riemann--Roch theorem, we have 
\begin{align*} 
\ch\left(p_\star(\cR \otimes \iota^\star(\L_1^{k_1} \otimes \cdots 
\otimes\L_r^{k_r}))\right) & = p_\star(\ch(\cR\otimes \L_1^{k_1}\otimes 
\cdots \otimes \L_r^{k_r}\otimes \ch(\T_p) ) \\
&= p_\star\left( e^{\mu_1 x_1+\cdots + \mu_r x_r} 
\prod_{i<j} \frac{x_i-x_j}{1 - e^{x_i-x_j}} \right) \\
& = p_\star\left(e^{k_1 x_1 + \cdots + k_r x_r} \prod_{i<j} 
\frac{x_i-x_j}{e^{x_j} - e^{x_i}} \right). 
\end{align*} 
By Lemma \ref{lem:p_star}, it follows that this is anti-symmetric in 
$k_1,\dots,k_r$. Therefore 
\begin{align*} 
\frac{1}{r!} 
\ch\left(p_\star(\cR\otimes \iota^\star(\L_1^{k_1} \wedge 
\cdots \wedge \L_r^{k_r}))\right ) 
& = 
\ch\left(p_\star(\cR\otimes \iota^\star(\L_1^{k_1}\otimes 
\cdots \otimes \L_r^{k_r}))\right) \\
& = \ch \left(p_\star(\L_1^{\mu_1} \otimes 
\cdots \otimes \L_1^{\mu_r}) \right) 
= \ch(\E_\mu).  
\end{align*} 
This shows that $(r!)^{-1}p_\star(\cR \otimes \iota^\star(\L_1^{k_1} 
\wedge \cdots \wedge \L_r^{k_r}))= \E_\mu$. 

(3): It suffices to prove the commutativity of the following two diagrams: 
\[
\xymatrix{
\wedge^r K^0(\P^{n-1}) \ar[r]^{\phantom{AB}\iota^\star} 
\ar[d]^{\Gg_\P \Ch(\cdot)} 
& 
K^0(\F) \ar[d]^{(\iota^\star\Gg_\P) \Ch(\cdot)} \\
\wedge^r H^0(\P^{r-1}) \ar[r]^{\phantom{AB}\iota^\star} 
& 
H^0(\G) 
} 
\quad 
\xymatrix{
K^0(\F) \ar[rrr]^{p_\star (\cR \otimes (\cdot))} 
\ar[d]^{(\iota^\star \Gg_\P) \Ch(\cdot)}
& & & 
K^0(\G) 
\ar[d]^{\Gg_\G \Ch(\cdot)} 
\\ 
H^\udot(\F) 
\ar[rrr]^{\epsilon_re^{-\pi\iu(r-1)\sigma_1}p_\star}
& & & 
H^\udot(\G)}
\]
The commutativity of the left diagram is obvious. 
The commutativity of the right diagram follows 
from Grothendieck--Riemann--Roch. 
For $\alpha \in K^0(\F)$, we have 
\begin{align*} 
\Ch(p_\star(\cR\otimes \alpha)) = 
\epsilon_r p_\star(\Ch(\cR\otimes \alpha)\Td(\T_p)). 
\end{align*} 
From $\iota^\star T\P \cong p^\star T\G \oplus 
\T_p \oplus \ov{\T_p}$ (Lemma \ref{lem:tangentbundles} 
and Remark \ref{rem:complexification}), we have 
\[
\iota^\star\Gg_\P = (p^\star \Gg_\P) \Gg(\T_p) \Gg(\T_p^\vee) 
= (p^\star \Gg_\P) \Td(\T_p) e^{-\pi\iu c_1(\T_p)}   
\] 
where we used the fact that the Gamma class is a square root 
of the Todd class (see \S\ref{subsec:HRR}). 
Therefore 
\begin{align*} 
\Gg_\G \Ch(p_\star(\cR\otimes \alpha)) 
& = \epsilon_r p_\star\left( 
(p^\star\Gg_\G) \Ch(\cR\otimes \alpha) \Td(\T_p) \right) \\ 
& = \epsilon_r p_\star\left( (\iota^\star \Gg_\P) \Ch(\alpha) 
\Ch(\cR) e^{-\pi\iu c_1(\T_p)}\right).  
\end{align*}
The relationship \eqref{eq:half-canonical} gives 
$\Ch(\cR) e^{-\pi\iu c_1(\T_p)} 
= e^{-\pi\iu(r-1) \sigma_1}$ and the 
commutativity follows. 
\end{proof} 

We define the Euler pairing on $\wedge^r K^0(\P^{n-1})$ as 
\[
\chi(\alpha_1\wedge \alpha_2\wedge \cdots \wedge \alpha_r, 
\beta_1\wedge \beta_2 \wedge \cdots \wedge \beta_r) 
:= \det(\chi_{\P^{n-1}}(\alpha_i,\beta_j))_{1\le i,j\le r}. 
\]
This is $1/r!$ of the Euler pairing induced from $K^0(\P)$. 
Recall the non-symmetric pairing $[\cdot,\cdot)$ defined in 
\eqref{eq:pairing_[)}. 
Similarly we define the pairing $[\cdot,\cdot)$ on $\wedge^r H^\udot(\P^{n-1})$ 
by 
\begin{equation} 
\label{eq:antisymmetrized_[)}
[\alpha_1\wedge \alpha_2\wedge \cdots \wedge \alpha_r, 
\beta_1\wedge \beta_2\wedge \cdots \wedge \beta_r)
:= \det([\alpha_i,\beta_j))_{1\le i,j\le r}. 
\end{equation}
This is again $1/r!$ of the pairing $[\cdot,\cdot)$ on $H^\udot(\P)$. 

\begin{proposition} 
\label{prop:pairing_ANA}
The horizontal maps in the diagram \eqref{eq:comparison_coh_K} 
preserves the pairings.  
More precisely, the map 
\[
(r!)^{-1} p_\star (\cR \otimes \iota^\star(\cdot)) 
\colon \wedge^r K^0(\P^{n-1}) \to K^0(\G)
\] 
preserves the Euler pairing $\chi$ and the map 
\[(r!)^{-1} \epsilon_r e^{-\pi\iu (r-1) \sigma_1} 
p_\star \iota^\star \colon \wedge^r H^\udot(\P^{n-1}) 
\to H^\udot(\G)
\] 
preserves the pairing $[\cdot,\cdot)$. 
\end{proposition} 
\begin{proof} 
Since the vertical maps in \eqref{eq:comparison_coh_K} intertwines 
the pairings $\chi$ and $[\cdot,\cdot)$ (see \eqref{eq:factorize_HRR}), 
it suffices to show that the map $(r!)^{-1} \epsilon_r e^{-\pi\iu(r-1)\sigma_1}
p_\star \iota^\star$ 
on cohomology preserves the pairing $[\cdot,\cdot)$. 
For $\alpha,\beta \in \wedge^r H^\udot(\P^{n-1})$, we have 
\begin{align*} 
& \left[ (r!)^{-1} \epsilon_r e^{-\pi\iu(r-1) \sigma_1} 
p_\star \iota^\star(\alpha), 
(r!)^{-1} \epsilon_r e^{-\pi\iu(r-1) \sigma_1} 
p_\star \iota^\star(\beta) \right) \\ 
& \qquad = \left(\frac{\epsilon_r}{r!}\right)^2 
\frac{1}{(2\pi\iu)^{\dim \G}}
\int_\G e^{\pi\iu c_1(\G)} 
(e^{\pi\iu \deg/2} p_\star\iota^\star\alpha) \cup 
p_\star \iota^\star(\beta) \\ 
& \qquad = \frac{1}{(r!)^2} 
\frac{(-1)^{\binom{r}{2}}}{(2\pi\iu)^{\dim \P}}
\int_\G 
(p_\star\iota^\star e^{\pi\iu c_1(\P)}e^{\pi\iu \deg/2}  \alpha) \cup 
p_\star \iota^\star(\beta) \\ 
& \qquad = \frac{1}{(r!)^2} 
\frac{(-1)^{\binom{r}{2}}}{(2\pi\iu)^{\dim \P}}
\int_\P
(e^{\pi\iu c_1(\P)}e^{\pi\iu \deg/2}  \alpha) \cup 
\iota_\star p^\star p_\star \iota^\star(\beta). 
\end{align*} 
By the formulas in Lemma \ref{lem:iota_star} and Lemma \ref{lem:p_star}, 
it follows easily that $\iota_\star p^\star p_\star \iota^\star \beta 
= r! (-1)^{\binom{r}{2}} \beta$ for the antisymmetric element $\beta$. 
Therefore this gives $1/r!$ of the pairing $[\cdot,\cdot)$ 
on $\P$ and the conclusion follows. 
\end{proof} 

\begin{corollary}[\cite{Ue05}] 
For $\mu_1\ge \mu_2 \ge \cdots \ge \mu_r$ 
and $\nu_1\ge \nu_2 \ge \cdots \ge \nu_r$ 
\[
\chi(\E_\mu, \E_\nu) = \det\left( 
\chi(\O_{\P^{n-1}}(l_i), 
\O_{\P^{n-1}}(k_i) )_{1\le i,j\le r} \right) 
\]
where $l_i = \mu_i + r- i$ and $k_i = \nu_i + r-i$. 
\end{corollary}

\begin{remark} 
The map $(r!)^{-1} p_\star \iota^\star \colon 
\wedge^r H^\udot(\P^{n-1}) \to H^\udot(\G)$ 
sending $x_1^{\mu_1+r-1} \wedge x_2^{\mu_2+r-2} 
\wedge \cdots \wedge x_r^{\mu_r}$ to $\sigma_\mu$ 
is called the \emph{Satake identification} in \cite{GGI:gammagrass}. 
\end{remark} 

\subsection{Quantum cohomology central charge} 
Here we restate the abelian/non-abelian correspondence 
of the $J$-functions in terms of 
quantum cohomology central charges. 
Recall from Remark \ref{rem:centralcharges} that the 
central charge $Z^\G(E)$ of a vector bundle $E\to \G$ 
is given by
\begin{equation}
\label{eq:Z} 
Z^\G(E)(t)= 
(2\pi\iu)^{\dim \G} 
\left[J_\G(e^{\pi\iu}t), \Gg_\G \Ch(E) \right), 
\end{equation} 
where $[\cdot,\cdot)$ is the pairing defined in \eqref{eq:pairing_[)}. 
This is a function on the universal cover of the punctured $t$-plane 
$\C^\times$. A branch of this function is determined when we 
specify $\arg t\in\R$. 
When $t\in \R_{>0}$, we regard $\arg t = 0$ unless otherwise 
specified.  
For $\alpha\in K^0(\P)$, 
we define 
\[
Z^\P(\alpha) := (2\pi\iu)^{\dim \P} 
\left[J_\P(e^{\pi\iu}t_1,\dots,e^{\pi\iu}t_r), 
\Gg_\P \Ch(\alpha) 
\right) 
\]
where $[\cdot,\cdot)$ is the pairing on $H^\udot(\P)$. 
When $\alpha = \alpha_1 \wedge \alpha_2 \wedge 
\cdots \wedge \alpha_r\in \wedge^r K^0(\P^{n-1})$, we have 
\[
Z^\P(\alpha) = \det\left(Z^{\P^{n-1}}(\alpha_i)(t_j)\right)_{1\le i,j\le r}
\]
where $Z^{\P^{n-1}}$ is the quantum cohomology 
central charge for $\P^{n-1}$. 

\begin{proposition}
\label{prop:cc_ANA} 
Let $Z^\G$ and $Z^{\P}$ denote the quantum cohomology 
central charges of $\G$ and $\P$ respectively. 
Then we have the equality: 
\begin{equation*}
Z^\G(p_\star (\cR\otimes\iota^\star \alpha)) 
= \frac{1}{ (2\pi\iu n)^{\binom{r}{2}}} 
\left[ \prod_{i<j} (\theta_i -\theta_j) \cdot 
Z^\P(\alpha) 
\right]_{t_1 = \cdots = t_r= \xi t},   
\end{equation*} 
where $\alpha \in \wedge^r K^0(\P^{n-1})$ 
and $\xi := e^{\pi\iu (r-1)/n}$. 
\end{proposition} 
\begin{proof}
By Lemma \ref{lem:p_star}, the abelian/non-abelian correspondence 
(Theorem \ref{thm:BCFK}) can be written in the form: 
\[
J_\G(t) = \frac{1}{r!} e^{-\pi\iu (r-1)\sigma_1}p_\star \iota^\star
\tJ(t) 
\]
with 
\[
\tJ(t) = 
\left[ n^{-\binom{r}{2}}\prod_{i<j} (\theta_i -\theta_j) 
J_\P(t_1,\dots,t_r) \right]_{t_1= \cdots =t_r = \xi t }. 
\]
By Proposition \ref{prop:comparison_coh_K} (3), we have 
\[
\Gg_\G \Ch(p_\star(\cR \otimes \iota^\star\alpha)) 
=  \epsilon_r e^{-\pi\iu (r-1) \sigma_1} 
p_\star \iota^\star \left(\Gg_\P \Ch(\alpha)\right).  
\]
By the above equations and the fact that $\frac{\epsilon_r}{r!}
e^{-\pi\iu(r-1)\sigma_1} p_\star \iota^\star$ 
preserves the pairing $[\cdot,\cdot)$ (Proposition \ref{prop:pairing_ANA}), 
we obtain 
\begin{align*} 
Z^\G(p_\star(\cR \otimes \iota^\star \alpha)) 
= (2\pi\iu)^{\dim \G+ \binom{r}{2}}
r! \left[ \tJ(t), 
\Gg_\P \Ch(\alpha) \right) 
\end{align*} 
where the pairing $[\cdot,\cdot)$ on the right-hand side 
denotes the pairing \eqref{eq:antisymmetrized_[)} on 
$\wedge^r H^\udot(\P^{n-1})$ 
(which is $1/r!$ of the pairing $[\cdot,\cdot)$ on $H^\udot(\P)$). 
Using $\dim G + \binom{r}{2} = \dim \P -\binom{r}{2}$, 
we obtain the formula in the proposition. 
\end{proof} 

\subsection{Integral representation of quantum cohomology central charges} 
In this section we give integral representations of quantum cohomology 
central charges in terms of the Hori--Vafa mirrors. 
Let $f\colon (\C^\times)^{n-1} \to \C$ be 
the mirror Laurent polynomial of $\P^{n-1}$ (see \S \ref{sec:heuristics}): 
\[
f(y) = y_1 + \cdots + y_{n-1} + \frac{1}{y_1 y_2 \cdots y_{n-1}}. 
\] 
The critical values of $f$ are given as 
\[
v_k := n e^{- 2\pi\iu k/n}, 
\quad k \in \Z/n \Z. 
\]
Let $\phi\in \R$ be an admissible phase for 
$\{v_0,v_1,\dots,v_{n-1}\}$ 
(see \S \ref{subsec:lift}). 
Let $\Gamma_k(\phi) \subset (\C^\times)^{n-1}$ denote 
the Lefschetz thimble for $f$ 
associated with the critical value $v_k$ and the vanishing path 
$v_k + \R_{\ge 0} e^{\iu\phi}$ (see Figure \ref{fig:vanishing_paths}). 
The Lefschetz thimble $\Gamma_k(\phi)$ is homeomorphic to  $\R^{n-1}$ and 
fibers over the straight half-line $f(\Gamma_k) = v_k + \R_{\ge 0} e^{\iu\phi}$; 
the fiber is a vanishing cycle 
homeomorphic to $S^{n-2}$. 
We write $\Gamma_k^\vee(\phi)$ for the ``opposite'' Lefschetz thimble 
associated to the critical value $v_k$ and the 
path $v_k - \R_{\ge 0} e^{\iu\phi}$. 
We choose orientation of Lefschetz thimbles such that 
$\sharp(\Gamma_k(\phi) \cap \Gamma_k^\vee(\phi)) 
= (-1)^{(n-1)(n-2)/2}$. 

\begin{figure}[htb] 
\begin{picture}(300,100)(0,5)
\put(150,40){\circle{60}}
\put(180,40){\makebox(0,0){$\bullet$}}
\put(165,66){\makebox(0,0){$\bullet$}}
\put(135,66){\makebox(0,0){$\bullet$}} 
\put(120,40){\makebox(0,0){$\bullet$}} 
\put(165,14){\makebox(0,0){$\bullet$}} 
\put(135,14){\makebox(0,0){$\bullet$}} 
\put(180,40){\line(3,1){50}}
\put(165,66){\line(3,1){50}}
\put(135,66){\line(3,1){75}} 
\put(120,40){\line(3,1){110}} 
\put(165,14){\line(3,1){70}} 
\put(135,14){\line(3,1){100}} 
\put(230,15){\vector(3,1){40}}
\put(280,30){\scriptsize admissible direction $\phi$} 
\end{picture}
\caption{Vanishing paths in the admissible direction $\phi$}
\label{fig:vanishing_paths}
\end{figure}
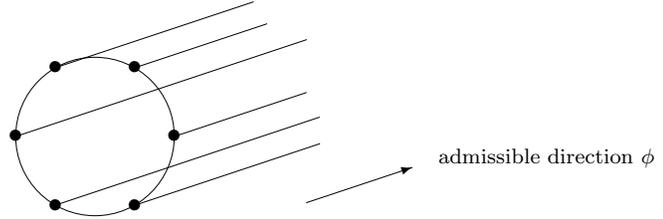

\begin{proposition}
\label{prop:cc_Pn-1} 
Let $\phi\in \R$ be an admissible phase 
for $\{v_0,v_1,\dots,v_{n-1}\}$. 
There exist $K$-classes $\cF_k(\phi) , \cG_k(\phi)  
\in K^0(\P^{n-1})$ 
for $k\in \Z/n\Z$ such that 
\begin{align}
\label{eq:FG_oscint}
\begin{split} 
Z^{\P^{n-1}}(\cF_k(\phi))(t) 
& =  
\int_{\Gamma_k(\phi)}  
e^{-tf(y)} \frac{dy}{y}, \\ 
Z^{\P^{n-1}}(\cG_k(\phi))(e^{-\pi\iu} t)  
& = 
\int_{\Gamma^\vee_k(\phi)} 
e^{tf(y)} \frac{dy}{y}, 
\end{split} 
\end{align} 
hold when $|\arg t + \phi|<\frac{\pi}{2}$, 
where $\frac{dy}{y}= \bigwedge_{i=1}^{n-1} \frac{dy_i}{y_i}$. 
Moreover, we have: 
\begin{enumerate}
\item $\chi(\cF_k(\phi),\cG_l(\phi)) = \delta_{kl}$; 
\item when $|k \frac{2\pi}{n} + \phi|<\frac{\pi}{2} + 
\frac{\pi}{n}$, 
we have $\cF_k(\phi) = \O_{\P^{n-1}}(k)$. 
\end{enumerate}
\end{proposition} 
\begin{proof} 
This follows from the result in \cite{Iritani09,KKP08}. 
(See \S\ref{sec:toric} and \cite[\S 5]{GGI:gammagrass} 
for a closely related discussion). 
By \cite[Theorems 4.11, 4.14]{Iritani09}, we have an 
isomorphism 
\[
K^0(\P^{n-1}) \cong H_{n-1}( (\C^\times)^{n-1}, 
\{ y : \re(t f(y)) \ge M\} ), \quad 
\alpha \mapsto \Gamma(\alpha,\arg t)  
\]
depending on $\arg t\in \R$, such that 
$\Gamma(\alpha,\arg t)$ is Gauss--Manin flat with respect to 
the variation of $\arg t$ and that 
\begin{equation*} 
Z^{\P^{n-1}}(\alpha)(t) = \int_{\Gamma(\alpha,\arg t)} e^{-t f(y)} \frac{dy}{y}. 
\end{equation*} 
Here $M$ is a sufficiently big positive number (it is sufficient that 
$M \ge 2n|t|$). 
Moreover, this isomorphism 
intertwines the Euler pairing with 
the intersection pairing:  
\[
\chi(\alpha,\beta) = (-1)^{\frac{(n-1)(n-2)}{2}}
\sharp\left( \Gamma(\alpha, \arg t + \pi) 
\cap \Gamma(\beta, \arg t)\right).  
\]
Therefore equation \eqref{eq:FG_oscint} and part (1) of the 
proposition hold for $\cF_k, \cG_k$ such that 
\[
\Gamma(\cF_k(\phi),-\phi) = \Gamma_k(\phi), \quad 
\Gamma(\cG_k(\phi),-\phi - \pi) = \Gamma_k^\vee(\phi). 
\]
(We only need to check \eqref{eq:FG_oscint} at 
$\arg t = -\phi$; it then follows by analytic continuation 
for other $t\in (-\phi-\frac{\pi}{2},-\phi+\frac{\pi}{2})$.) 
Recall from \eqref{eq:oscint_centralcharge} that we have 
\[
Z^{\P^{n-1}}(\O_{\P^{n-1}})(t) 
= \int_{\Gamma_0(0)} e^{-tf(y)} \frac{dy}{y} 
\]
when $\arg t =0$. 
Therefore, when $\arg t = 2\pi k/n$, we have 
by setting $t' = e^{-2\pi \iu k/n} t$, 
\begin{align*} 
Z^{\P^{n-1}}(\O_{\P^{n-1}}(k))(t)
& = Z^{\P^{n-1}}(\O_{\P^{n-1}})(t') \\
& = \int_{\Gamma_0(0)} e^{-t'  f(y)} \frac{dy}{y} 
= \int_{\Gamma_k(-2\pi k/n)} 
e^{-t f(y)} \frac{dy}{y} 
\end{align*} 
where in the first line we used \eqref{eq:J-function_P} 
and the definition of the central charge, and 
in the second line we performed the change 
$y_i \to e^{2\pi\iu k/n} y_i$ of variables 
and used the fact that $e^{-2\pi\iu k/n} \Gamma_0(0) 
= \Gamma_k(-2\pi k/n)$. 
From this it follows that 
\[
Z^{\P^{n-1}}(\O_{\P^{n-1}}(k))(t) = 
\int_{\Gamma_k(\phi)} e^{-t f(y)} \frac{dy}{y} 
\]
when 
$|\phi+ k \frac{2 \pi}{n}|< \frac{\pi}{2} + \frac{\pi}{n}$ 
and $\arg t = -\phi$.  
This implies $\cF_k(\phi) = \O_{\P^{n-1}}(k)$ 
for such $\phi$ and $k$ and part (2) follows. 
%
%
%
%
%
%
\end{proof} 

\begin{remark} 
The sign $(-1)^{(n-1)(n-2)/2}$ in the intersection pairing 
was missing in \cite{Iritani09}, and this has been corrected in 
\cite[footnote 16]{Iritani11}.
\end{remark}

\begin{remark} 
\label{rem:exceptional_collection_P}
Part (2) of the proposition determines roughly half of $\cF_k(\phi)$'s. 
The whole collection $\{\cF_k(\phi) : k\in \Z/n\Z \}$ is obtained from the 
exceptional collection $\{ \O(k) : -\phi-\pi < k \frac{2\pi}{n} <-\phi+\pi\}$ 
by a sequence of mutations (see \cite[\S 5]{GGI:gammagrass}), 
and thus is an exceptional collection. 
Because $\cF_k(\phi)$ is mirror to the Lefschetz thimble $\Gamma_k(\phi)$, 
the set $\{\Gg_{\P^{n-1}}\Ch(\cF_k(\phi)): k\in \Z/n\Z\}$ 
gives the asymptotic basis of $\P^{n-1}$ 
at phase $\phi$ (see the proof of Theorem \ref{thm:toric_GammaII})  
and the Gamma conjecture II (\S \ref{subsec:asymp_basis_GII}) 
holds for $\P^{n-1}$. 
\end{remark}

The Hori--Vafa mirror $\tg$ of $\P = (\P^{n-1})^r $ is given by: 
\begin{align*} 
\tg(y) & := f(\vec{y}_1) + f(\vec{y}_2) + \cdots + f(\vec{y}_r) \\ 
& = \sum_{i=1}^r \left(y_{i,1} + y_{i,2} + \cdots + 
y_{i,n-1} + \frac{1}{y_{i,1} y_{i,2} \cdots y_{i,n-1}} 
\right) 
\end{align*} 
where $y =(\vec{y}_1,\dots,\vec{y}_r) \in (\C^\times)^{r(n-1)}$ 
and $\vec{y}_i = (y_{i,1},\dots,y_{i,n-1})\in (\C^\times)^{n-1}$. 
Critical values of $\tg$ are given by 
\[
\tv_K:= 
v_{k_1} + \cdots + v_{k_r} \quad \text{for } K=(k_1,\dots, k_r) 
\in \Z^r.  
\]
The product
$\Gamma_{k_1}(\phi) \times \cdots \times \Gamma_{k_r}(\phi)$
of Lefschetz thimbles for $f$ 
gives a Lefschetz thimble for $\tg$ associated to 
the critical value $\tv_K$ 
and the path $\tv_K + \R_{\ge 0} e^{\iu\phi}$.  

The Hori--Vafa mirror of $\G$ is obtained 
from the mirror of $\P$ by shifting the phase by $(r-1)\pi/n$ and 
restricting to ``anti-symmetrized''  Lefschetz thimbles 
(see \cite{Hori-Vafa, Ue05,KimSab}). 
We set 
\[
g(y) := \xi \tg(y) 
\]
with $\xi = e^{\pi \iu (r-1)/n}$. 
Then critical values of $g(y)$ are 
\[
v_K := \xi \tv_K 
\quad \text{for }K = (k_1,\dots,k_r) \in \Z^r. 
\]
For a tuple $K = (k_1,\dots,k_r)\in \Z^r$ and $\phi \in \R$, 
we define the anti-symmetrized Lefschetz thimble for $g(y)$ as: 
\begin{align*} 
\Gamma_K(\phi) & := \Gamma_{k_1}(\phi') 
\wedge \cdots \wedge \Gamma_{k_r}(\phi') 
= \sum_{\sigma\in \frS_r} 
\sgn(\sigma) \Gamma_{k_{\sigma(1)}}(\phi') 
\times \cdots \times \Gamma_{k_{\sigma(r)}}(\phi') \\ 
\Gamma^\vee_K(\phi) & := 
\Gamma^\vee_{k_1}(\phi') \wedge \cdots 
\wedge \Gamma^\vee_{k_r}(\phi') 
= \sum_{\sigma \in \frS_r} 
\sgn(\sigma) \Gamma_{k_{\sigma(1)}}^\vee(\phi') 
\times \cdots \times \Gamma_{k_{\sigma(r)}}^\vee(\phi')  
\end{align*} 
with $\phi' = \phi - \frac{(r-1)\pi}{n}$. 
They are elements of the relative homology groups 
\[
H_{r(n-1)}((\C^\times)^{r(n-1)}, \{y : \pm \re(e^{-\iu\phi}g(y)) \ge M\})
\]
with $M$ sufficiently large. 
We also define $K$-classes $\cF_K(\phi), \cG_K(\phi) \in K^0(\G)$ as:  
\begin{align} 
\label{eq:FG_K} 
\begin{split} 
\cF_K(\phi) & := \frac{1}{r!} p_\star\left(\cR \otimes \iota^\star 
(\cF_{k_1}(\phi') \wedge 
\cdots \wedge \cF_{k_r}(\phi')) \right) \\ 
\cG_K(\phi) & := \frac{1}{r!} p_\star \left(\cR \otimes \iota^\star 
(\cG_{k_1}(\phi') \wedge \cdots \wedge \cG_{k_r}(\phi')) 
\right)
\end{split} 
\end{align} 
with $\phi' = \phi - \frac{(r-1)\pi}{n}$, where $\cF_k(\phi)$ and 
$\cG_k(\phi)$ are as in Proposition \ref{prop:cc_Pn-1}. 
Note that $\{\cF_K(\phi)\}_K$ or $\{\cG_K(\phi)\}_K$ gives 
an integral basis of $K^0(\G)$ by Proposition \ref{prop:comparison_coh_K}. 
\begin{remark} 
By the change of variables $y_{i,j} = \xi y'_{i,j}$, 
we can also write: 
\[
g(y) = \sum_{i=1}^r \left( y'_{i,1} + y'_{i,2} + 
\cdots + y'_{i,n-1} + \frac{(-1)^{r-1}}{y'_{i,1}y'_{i,2} \cdots y'_{i,n-1}} 
\right) 
\]
\end{remark} 

\begin{theorem} 
\label{thm:oscint_G}
Let $\phi\in \R$ be such that $\phi' = \phi- \frac{(r-1)\pi}{n}$ 
is admissible for $\{v_0,v_1,\dots,v_{n-1}\}$. 
For a mutually distinct $r$-tuple 
$K = (k_1,\dots,k_r) \in (\Z/n\Z)^r$, we have  
\begin{align*} 
Z^\G(\cF_K(\phi))(t) & = \frac{1}{r!} 
\left(\frac{-\xi t}{2\pi\iu n}\right)^{\binom{r}{2}}
\int_{\Gamma_K(\phi)} \frac{dy}{y} e^{-t g(y)} 
\cdot 
\prod_{i<j} (f(\vec{y}_i) - f(\vec{y}_j))  \\ 
Z^\G(\cG_K(\phi))(e^{-\pi\iu} t) & = 
\frac{1}{r!} 
\left(\frac{\xi t}{2\pi\iu n}\right)^{\binom{r}{2}}
\int_{\Gamma^\vee_K(\phi)} 
\frac{dy}{y} e^{t g(y)} \cdot 
\prod_{i<j} (f(\vec{y}_i) - f(\vec{y}_j)) 
\end{align*} 
for $|\arg t + \phi|<\frac{\pi}{2}$,  
where $\xi := e^{\pi\iu(r-1)/n}$ and 
$\frac{dy}{y} := \bigwedge_{i=1}^r \bigwedge_{j=1}^{n-1}
\frac{dy_{i,j}}{y_{i,j}}$. 
Moreover, we have 
\[
\chi(\cF_K(\phi), \cG_L(\phi)) =\delta_{K,L} 
\]
when the tuples $K$, $L$ are ordered with respect to 
a fixed choice of a total order of $\Z/n\Z$. 
\end{theorem} 
\begin{proof} 
Combining Proposition~\ref{prop:cc_ANA} and 
Proposition~\ref{prop:cc_Pn-1}, we have that 
\begin{align*} 
Z^\G(\cF_K(\phi))(t) & = 
\frac{1}{r! (2\pi\iu n)^{\binom{r}{2}}} 
\left[ \prod_{i<j}(\theta_i - \theta_j) \cdot 
Z^\P(\cF_{k_1}(\phi') \wedge \cdots \wedge \cF_{k_r}(\phi') ) 
\right]_{t_1=\cdots = t_r = \xi t} \\
& = \frac{1}{r! (2\pi\iu n)^{\binom{r}{2}}} 
\left[ \prod_{i<j}(\theta_i - \theta_j) \cdot 
\int_{\Gamma_K(\phi)} 
e^{- (t_1 f(\vec{y}_1)  + \cdots +  t_r f(\vec{y}_r)) }
\frac{dy}{y} \right]_{t_1 = \cdots = t_r = \xi t} \\
& 
= \frac{1}{r!}\left( \frac{-\xi t}{2\pi\iu n}\right)^{\binom{r}{2}} 
\int_{\Gamma_K(\phi)} 
\frac{dy}{y}
e^{-t g(y)} \prod_{i<j} (f(\vec{y}_i) - f(\vec{y}_j)), 
\end{align*} 
where in the last line, we used Lemma \ref{lem:antisym_derivation} 
below. 
The formula for $Z^\G(\cG_K(\phi))(e^{-\pi\iu}t)$ follows similarly. 
The orthogonality $\chi(\cF_K(\phi), \cG_L(\phi)) = \delta_{K,L}$ 
follows easily from Propositions \ref{prop:pairing_ANA} 
and \ref{prop:cc_Pn-1}. 
\end{proof} 

\begin{lemma} 
\label{lem:antisym_derivation} 
We have 
\[
\left(\prod_{i<j} (\theta_i - \theta_j) \right) e^{\sum_{i=1}^r \alpha_i t_i} 
=  e^{\sum_{i=1}^r \alpha_i t_i} \prod_{i<j}(\alpha_i t_i -\alpha_j t_j)
\]
where $\theta_i = t_i\parfrac{}{t_i}$. 
\end{lemma} 
\begin{proof} 
The left-hand side can be written in 
the form $\varphi(\alpha_1t_1,\dots,\alpha_r t_r) e^{\sum_{i=1}^r \alpha_i t_i}$ 
for some polynomial $\varphi$ and the highest order term of $\varphi$ 
is $\prod_{i<j} (\alpha_i t_i - \alpha_j t_j)$. 
On the other hand, $\varphi$ is anti-symmetric 
in the arguments, and thus should be divisible 
by $\prod_{i<j} (\alpha_i t_i - \alpha_j t_j)$. 
This implies the lemma. 
\end{proof} 

\begin{remark} 
The collection $\{\cF_K(\phi)\}_K$, where $K$ ranges over distinct 
$r$ elements of $\Z/n\Z$, yields the asymptotic basis 
$\{\Gg_\G \Ch(\cF_K(\phi))\}_K$ of $\G$ at phase $\phi$. 
This follows either by studying mirror oscillatory integrals  
in more details or by combining \cite[Proposition 6.5.1]{GGI:gammagrass}, 
Remark \ref{rem:exceptional_collection_P} 
and Proposition \ref{prop:comparison_coh_K}. 
It follows from the deformation argument in \cite[\S 6]{GGI:gammagrass} 
that this is mutation equivalent to Kapranov's exceptional 
collection $\{\E_\mu: n-r \ge \mu_1 \ge \cdots \ge \mu_r \ge 0\}$ 
\cite{Kapranov83}. 
\end{remark} 

\subsection{Quantum periods via mirrors of Eguchi--Hori--Xiong, Rietsch and 
Marsh--Rietsch} 
In this section, using mirrors of Eguchi--Hori--Xiong \cite{EHX}, 
Rietsch \cite{Rietsch} and Marsh--Rietsch \cite{Marsh-Rietsch}, 
we show that the limit sup in \eqref{eq:Cauchy-Hadamard} can be 
replaced with the limit for Grassmannians. 

Eguchi--Hori--Xiong's mirror \cite{EHX} for Grassmannians 
is a Laurent polynomial in $r(n-r) = \dim \G$ variables. 
Let $X_{i,j}$ with $1\le i\le r$, $1\le j\le n-r$ be $r(n-r)$ independent 
variables. 
Define the Laurent polynomial $W$ in these 
variables by 
\[
W= \sum_{i=1}^r \sum_{j=1}^{n-r-1} 
\frac{X_{i,{j+1}}}{X_{i,j}} 
+ \sum_{j=1}^{n-r} \sum_{i=1}^{r-1} 
\frac{X_{i+1,j}}{X_{i,j}} 
+ \frac{1}{X_{1,n-r}} + \frac{1}{X_{r,1}}. 
\]
Batyrev--Ciocan-Fontanine--Kim--van-Straten 
\cite{BCFKvS} showed that the Newton polytope 
of $W$ equals the fan polytope of a toric degeneration 
of $\G = \Gr(r,n)$ and conjectured that 
$W$ is a weak Landau--Ginzburg model of $\G$. 
Recently, Marsh--Rietsch \cite{Marsh-Rietsch} proved this conjecture 
by constructing a compactification of the Eguchi--Hori--Xiong mirror 
and showing an isomorphism between the quantum connection 
and the Gauss--Manin connection. Indeed, the mirror of Marsh--Rietsch 
is a reinterpretation of the earlier construction by Rietsch \cite{Rietsch}. 
\begin{theorem}[\cite{Marsh-Rietsch}] 
The quantum period \eqref{eq:qperiod} 
$G_\G(t) = \Ang{[\pt],J_\G(t)}$ of the Grassmannian 
$\G$ is given by the constant term series of the Eguchi--Hori--Xiong mirror 
$W$, i.e.~$G_\G(t) = \sum_{i=0}^\infty \frac{1}{i!} 
\Const(W^i) t^i$. 
\end{theorem}

It is easy to check that the constant term series of $W$ is of 
the form $\sum_{k=0}^\infty a_k t^{kn}$ with $a_k \neq 0$ 
for all $k\in \Z_{\ge 0}$ (see \cite[\S 5.2]{BCFKvS} for the explicit form). 
Therefore, by applying Lemma \ref{lem:limsup_lim},  we obtain 
the following. (It seems difficult to deduce this from 
Hori--Vafa mirrors). 

\begin{proposition} 
\label{prop:G_qperiod}
Let $G_\G(t) = \Ang{[\pt], J_\G(t)} = \sum_{k=0}^\infty G_{kn} t^{kn}$ 
be the quantum period of $\G$. Then $G_{kn}>0$ for all $k\in \Z_{\ge 0}$ 
and $\lim_{k\to\infty} \sqrt[kn]{(kn)! G_{kn}}$ exists. 
\end{proposition}

\subsection{Ap\'ery constants}
In this section we prove 
Gamma conjecture I' (Conjecture \ref{conj:I_dash}) 
for $\G = \Gr(r,n)$. 

\begin{theorem} 
\label{thm:Grass_I_dash} 
The Grassmannian $\Gr(r,n)$ satisfies Gamma conjecture I'. 
\end{theorem} 

The rest of the paper is devoted to the proof of Theorem \ref{thm:Grass_I_dash}. 
Since $\Gr(r,n) \cong \Gr(n-r, n)$, we may assume that 
$r \le n/2$. 
We fix a sufficiently small phase $\phi>0$ 
such that $\phi' = \phi - \frac{(r-1)\pi}{n}$ 
is admissible for $\{v_0,v_1,\dots,v_{n-1}\}$. 
Let $\Lambda$ denote the index set of 
mutually distinct $r$-tuples $K$ of elements of $\Z/n\Z$: 
\[
\Lambda= \{K = (k_1,\dots,k_r) \in (\Z/n\Z)^r 
: n-1 \ge k_1 > k_2 > \cdots > k_r \ge 0\}.  
\]
For $K\in \Lambda$, we write $\cF_K = \cF_K(\phi)$ 
and $\cG_K = \cG_K(\phi)$ for the elements in \eqref{eq:FG_K}. 

\begin{lemma}
\label{lem:longest}
Set $T := \max\{ |v_K| : K \in \Lambda\}$.  
\begin{enumerate}
\item We have
$
T= n \frac{\sin(\pi r/n)}{\sin(\pi/n)}$. 
\item If $T = |v_K|$ for $K\in \Lambda$, then $K$ is given 
by a consecutive $r$-tuple of elements in $\Z/n\Z$ 
and $v_K = T e^{-2\pi\iu k/n}$ for some $k\in \Z$. 
\item We have $T = v_{K_0}$ for $K_0 := (r-1,r-2,\dots,1,0)$. 
Moreover, $\cF_{K_0} = \O_\G$.  
\end{enumerate}
\end{lemma} 
\begin{proof} 
Parts (1) and (2) follow easily from the definition of $v_K$. 
To see Part (3), note that Proposition \ref{prop:cc_Pn-1} (2) 
gives $\cF_k(\phi') = \O_{\P^{n-1}}(k)$ for $0\le k\le r-1$ 
and $\phi' = \phi - \frac{(r-1)\pi}{n}$ (since $|\phi|$ is small 
and $r\le n/2$). 
Then Proposition \ref{prop:comparison_coh_K} (2) implies the conclusion. 
\end{proof}

Let $\alpha \in H_{2|\alpha|}(\G)$ be a homology class 
such that $c_1(\G) \cap \alpha =0$. 
We want to show that the limit formula (see \eqref{eq:Apery_limit}) 
\[
\lim_{n \to \infty} \frac{\Ang{\alpha,J_{kn}}}{
\Ang{[\pt],J_{kn}}} = \langle \alpha, \Gg_\G \rangle 
\]
holds where we set $J_\G(t) = e^{c_1(\G) \log t} \sum_{k=0}^\infty 
J_{kn} t^{kn}$. 
We start with noting that it suffices to show this formula 
when $\langle \alpha, \Gg_\G \rangle =0$. 
Indeed, we obtain the general case 
by applying the formula to 
$\alpha' = \alpha - \langle \alpha, \Gg_\G \rangle [\pt]$ 
(which satisfies $\langle \alpha', \Gg_\G \rangle =0$). 

Let $\halpha \in K^0(\G) \otimes \C$ be a complexified $K$-class 
such that 
\[
\PD(\alpha) = \Gg_\G \Ch(\halpha). 
\]
Then we have, using the definition of $Z^\G$ (see \eqref{eq:Z} and   
\eqref{eq:pairing_[)}), 
\begin{align*} 
\Ang{\alpha,J_\G(t)} & = \int_{\G}  J_\G(t) \cup \PD(\alpha) \\ 
& = (\iu)^{\dim \G - 2|\alpha|} 
\int_\G J_\G(t) \cup  e^{-\pi\iu\mu}e^{\pi \iu c_1(\G)} 
\PD(\alpha)  \\
& =  (-1)^{|\alpha|} (2\pi\iu)^{\dim \G} 
\left[J_\G(t), \Gg_\G \Ch(\halpha)\right) 
= (-1)^{|\alpha|} Z^\G(\halpha)(e^{-\pi\iu} t).  
\end{align*} 
Using the dual bases $\{\cF_K\}_{K\in \Lambda}$ 
and $\{\cG_K\}_{K\in \Lambda}$ of $K^0(\G)$, we can expand 
\[
\halpha = \sum_{K\in \Lambda} \chi(\cF_K, \halpha) \cG_K.  
\]
Thus by Theorem \ref{thm:oscint_G}, we obtain the integral 
representation for $\Ang{\alpha, J_\G(t)}$: 
\begin{equation} 
\label{eq:alpha_component_J}
\Ang{\alpha,J_\G(t)} = (-1)^{|\alpha|} c_{n,r}
\sum_{K\in \Lambda} \chi(\cF_K,\halpha) \cdot 
 t^{\binom{r}{2}} \int_{\Gamma^\vee_K(\phi)} 
\frac{dy}{y}  e^{t g(y)} \prod_{i<j} (f(\vec{y}_i) - f(\vec{y}_j)) 
\end{equation} 
for $|\arg t + \phi|<\pi/2$, 
where we set $c_{n,r} = \frac{1}{r!} (\xi/(2\pi \iu n))^{\binom{r}{2}}$.  
Let $C_K(\lambda)$ denote the ``anti-symmetrized'' vanishing cycle 
in the fiber $g^{-1}(\lambda)$: 
\[
C_K(\lambda) = \Gamma^\vee_K(\phi) \cap g^{-1}(\lambda)
= \sum_{\sigma \in \frS_r} 
\sgn(\sigma) C_{k_{\sigma(1)}}(\lambda) \times \cdots 
\times C_{k_{\sigma(r)}}(\lambda)  
\]
where $\lambda \in v_K - \R_{\ge 0} e^{\iu\phi}$ 
and $C_{k_i}(\lambda) =\Gamma_{k_i} (\phi') \cap 
f^{-1}(\xi^{-1} \lambda)$ is the vanishing cycle for $f$. 
We define the period integral $P_K(\lambda)$ as: 
\[
P_K(\lambda) := \int_{C_K(\lambda) \subset g^{-1}(\lambda)} 
\prod_{i<j} (f(\vec{y}_i) - f(\vec{y}_j)) 
\left. \frac{\bigwedge_{i=1}^r \bigwedge_{j=1}^{n-1} 
\frac{dy_{i,j}}{y_{i,j}}}{d g} \right|_{g^{-1}(\lambda)}. 
\]
Then we may rewrite \eqref{eq:alpha_component_J} as 
the Laplace transform of the period: 
\begin{equation} 
\label{eq:alpha_component_J_Laplaceperiod} 
\Ang{\alpha,J_\G(t)} = (-1)^{|\alpha|} 
c_{n,r} \sum_{K\in \Lambda} \chi(\cF_K, \halpha)  
\int_{v_K -e^{\iu\phi} \R_{\ge 0}} d\lambda 
\cdot t^{\binom{r}{2}}e^{t\lambda} P_K(\lambda). 
\end{equation} 
Applying the Laplace transformation 
\[
\varphi(t) \mapsto \int_0^\infty \varphi(t) e^{-u t} dt 
\]
to both sides of \eqref{eq:alpha_component_J_Laplaceperiod}, 
we obtain 
\begin{equation} 
\label{eq:Hilbert_trans} 
\sum_{k=0}^\infty (kn)! \Ang{\alpha,J_{kn}} u^{-kn-1}  
= 
(-1)^{|\alpha|} 
c_{n,r}' \sum_{K\in \Lambda} \chi(\cF_K, \halpha)  
\int_{v_K -e^{\iu\phi} \R_{\ge 0}} d\lambda 
\frac{P_K(\lambda)}{(u-\lambda)^{\binom{r}{2}+1}}
\end{equation} 
when $\re(u-v_K)>0$ for all $K\in \Lambda$, 
where $c_{n,r}' = \binom{r}{2}! c_{n,r}$. 
Note that the right-hand side can be analytically continued 
to a holomorphic function outside the branch cut: 
\[
\bigcup_{K\in \Lambda} v_K - e^{\iu\phi} \R_{\ge 0}. 
\]
See Figure \ref{fig:branchcut}. 
Moreover, this can be analytically continued to the universal 
cover of $\C \setminus \{ v_K : K \in \Lambda\}$. 
\begin{figure}[htb]
\begin{picture}(300,120)(50,0)

\put(117.3,60){\makebox(0,0){$\bullet$}} 
\put(100,70){\makebox(0,0){$\bullet$}} 
\put(82.6,60){\makebox(0,0){$\bullet$}}
\put(82.6,40){\makebox(0,0){$\bullet$}}
\put(100,30){\makebox(0,0){$\bullet$}}
\put(117.3,40){\makebox(0,0){$\bullet$}}

\put(134.6,50){\makebox(0,0){$\bullet$}}
\put(117.3,80){\makebox(0,0){$\bullet$}}
\put(82.6,80){\makebox(0,0){$\bullet$}}
\put(65.4,50){\makebox(0,0){$\bullet$}}
\put(82.6,20){\makebox(0,0){$\bullet$}}
\put(117.3,20){\makebox(0,0){$\bullet$}}

\put(100,50){\makebox(0,0){$\bullet$}}

\put(145,50){\makebox(0,0){$T$}}

\put(117.3,60){\line(-9,-1){107.3}} 
\put(100,70){\line(-9,-1){90}} 
\put(82.6,60){\line(-9,-1){72.6}}
\put(82.6,40){\line(-9,-1){72.6}}
\put(100,30){\line(-9,-1){90}}
\put(117.3,40){\line(-9,-1){107.3}}

\put(134.6,50){\line(-9,-1){123.6}}
\put(117.3,80){\line(-9,-1){107.3}}
\put(82.6,80){\line(-9,-1){72.6}}
\put(65.4,50){\line(-9,-1){55.4}}
\put(82.6,20){\line(-9,-1){72.6}}
\put(117.3,20){\line(-9,-1){107.3}}

\put(100,50){\line(-9,-1){90}}


\put(330,50){\makebox(0,0){$\bullet$}}
\put(315,76){\makebox(0,0){$\bullet$}}
\put(285,76){\makebox(0,0){$\bullet$}} 
\put(270,50){\makebox(0,0){$\bullet$}} 
\put(315,24){\makebox(0,0){$\bullet$}} 
\put(285,24){\makebox(0,0){$\bullet$}} 

\put(345,76){\makebox(0,0){$\bullet$}}
\put(300,102){\makebox(0,0){$\bullet$}} 
\put(255,76){\makebox(0,0){$\bullet$}} 
\put(255,25){\makebox(0,0){$\bullet$}} 
\put(300,-2){\makebox(0,0){$\bullet$}} 
\put(345,24){\makebox(0,0){$\bullet$}} 

\put(330,57){\makebox(0,0){\scriptsize $1/T$}}
\put(300,55){\makebox(0,0){\scriptsize $0$}}

\put(330,50){\line(1,0){80}}
\put(300,102){\arc{103}{-0.3}{0.5}}
\put(300,67.3){\arc{34.6}{-4.71}{-0.52}}
\put(300,76){\arc{52}{-4.71}{-1.57}}
\put(300,102){\arc{103}{-4.71}{-2.9}}
\put(270,50){\line(1,0){30}}
\put(270,50){\line(-1,0){60}}

\put(300,-2){\arc{103}{-3.25}{-1.57}}
\put(300,24){\arc{52}{-4.71}{-1.57}}
\put(300,32.7){\arc{34.6}{-5.75}{-1.57}}
\put(300,-2){\arc{103}{-0.52}{0.1}}



\end{picture} 
\caption{Branch cut in the $u$-plane (left) 
and the $u^{-1}$-plane (right)  
($n=6,r=2$). In the right picture, we set $\phi=0$ for simplicity.}
\label{fig:branchcut} 
\end{figure}
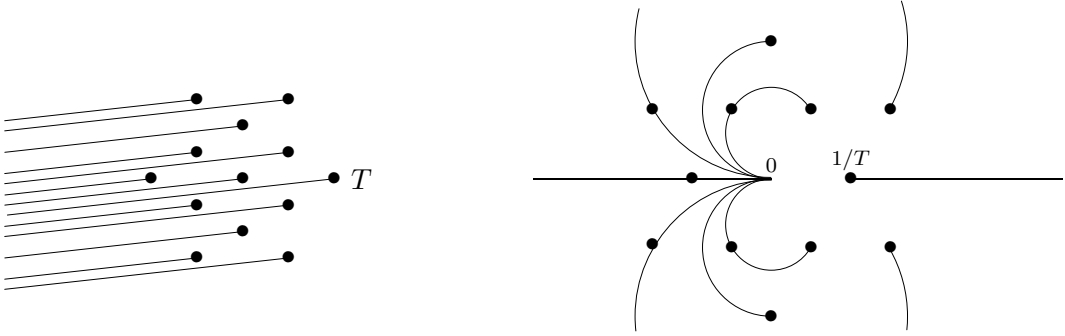
Since the left-hand side is regular at $u=\infty$, 
it follows from Lemma \ref{lem:longest} that 
the convergence radius of the left-hand side of \eqref{eq:Hilbert_trans} 
(as a power series in $u^{-1}$) 
is bigger than or equal to 
\[
\frac{1}{T} = \frac{1}{n} \frac{\sin(\pi/n)}{\sin(\pi r/n)}. 
\]
By Cauchy's integral formula, it follows that the jump of 
the function \eqref{eq:Hilbert_trans} across the cut 
$v_K - e^{\iu\phi} \R_{\ge 0}$ is proportional to 
$\chi(\cF_K,\halpha) \cdot 
(\partial_u)^{\binom{r}{2}} P_K(u)$. 
From this it follows that: 
\begin{itemize} 
\item[(a)] if $\alpha = [\pt]$, we have 
$\chi(\cF_{K_0}, \halpha) = (2\pi\iu)^{-\dim \G} 
\chi(\O_\G, \O_{\pt}) \neq 0$ 
and thus the convergence radius of $\sum_{k=0}^\infty 
(kn)! \Ang{[\pt],J_{kn}} x^{kn}$ is exactly $1/T$; 

\item[(b)] if $\langle \alpha, \Gg_\G \rangle =0$, then 
we have $\chi(\cF_{K_0}, \halpha) = \chi(\O_\G, \halpha)
 = [\Gg_\G, \alpha) = 0$ (by \eqref{eq:factorize_HRR}); 
this implies that \eqref{eq:Hilbert_trans} is holomorphic at 
$u= v_{K_0} = T$; since the left-hand side of \eqref{eq:Hilbert_trans} 
is a power series in $u^{-n}$ (multiplied by $u^{-1}$), 
it is holomorphic at any other $v_K$ with $|v_K| =T$ 
and the series 
$\sum_{k=0}^\infty (kn)! 
\Ang{\alpha, J_{kn}} x^{kn}$ has a convergence radius 
$R(\alpha) > 1/T$. 

\end{itemize} 
Here $K_0 =(r-1,r-2,\dots,1,0)$ and recall from 
Lemma \ref{lem:longest} that $\cF_{K_0} = \O_\G$. 
From part (a) and Proposition \ref{prop:G_qperiod}, 
it follows that 
\[
\lim_{k \to \infty} \sqrt[kn]{(kn)! \Ang{[\pt],J_{kn}}} = T.  
\] 
From part (b), it follows that, when $\langle \alpha, \Gg_\G \rangle =0$, 
\[
\lim_{k\to \infty}
\frac{\Ang{\alpha,J_{kn}}}{\Ang{[\pt],J_{kn}}}
= \lim_{k\to \infty} 
\frac{(kn)! \Ang{\alpha,J_{kn}}}{(kn)! \Ang{[\pt],J_{kn}}} 
= 0. 
\]
This completes the proof of Theorem \ref{thm:Grass_I_dash}. 

\appendix
\section{Eigenvalues of quantum multiplication by $c_1(F)$ on odd cohomology}
\label{app:odd} 
In the main body of the text, we restrict to the even part of cohomology 
group and ignore the odd part. 
The restriction to the even part was partly for the sake of simplicity: 
in general, the big quantum product defines a super-commutative algebra
structure on the full cohomology group, and the big quantum connection 
\eqref{eq:big_qconn} should be treated in the formalism of supermanifolds 
(see \cite{Man99}). 
In this appendix, answering a question of an anonymous referee, 
we discuss Property $\O$ and Gamma conjectures on the 
\emph{full} cohomology group, and show that they are in fact equivalent to 
Property $\O$ and Gamma conjectures on the even part (discussed 
in the main body of the text). 
Moreover, we argue that it is sufficient to consider Property $\O$ and Gamma 
conjectures on the subspace $\bigoplus_{p=0}^{\dim F} H^{p,p}(F)$ 
of diagonal Hodge type. 
In this section, we write $H^{\rm even}(F)$/$H^{\rm odd}(F)$
for the even/odd part of cohomology group, and $H^{\rm full}(F) 
= H^{\rm even}(F) \oplus H^{\rm odd}(F)$ for the full cohomology group. 
(We wrote $H^\udot(F) = H^{\rm even}(F)$ in the main body of the text.) 

\begin{remark} 
After we wrote this appendix, we noticed that Sanda--Shamoto 
\cite[Remark 6.5, Lemma 6.6]{Sanda-Shamoto} had already discussed 
the same issue on odd cohomology. Our new observation is 
the discussion concerning $\bigoplus_{p=0}^{\dim F} H^{p,p}(F)$. 
\end{remark} 

For Gamma Conjecture II, the restriction to the even part was superfluous. 
Recall from Hertling--Manin--Teleman \cite{HMT} that if the quantum cohomology of $F$ 
is semisimple, $F$ is necessarily of Hodge-Tate type (i.e.~$H^{p,q}(F) = 0$ 
unless $p=q$), 
and in particular has no odd cohomology classes. 

Let us study Property $\O$ and Gamma Conjecture I including the odd part. 
For a subspace $V \subset H^{\rm full}(F)$ preserved by $(c_1(F)\star_0)$, 
we define 
\[
T_V := \max\{ |u| : 
\text{$u$ is an eigenvalue of $(c_1(F)\star_0) \colon V \to V$}\} 
\in \overline{\Q}. 
\]
The number $T$ \eqref{eq:T} in the main body of the text equals 
$T_{H^{\rm even}(F)}$. 
By replacing $T$ with $T_V$ in Definition \ref{def:propertyO} 
and regarding $(c_1(F)\star_0)$ as an operator on $V$, 
we can similarly define ``Property $\O$ \emph{on the subspace $V$}''.  
Note that the original Property $\O$ is the ``Property $\O$ on $H^{\rm even}(F)$'' 
in this terminology. 
We define 
\[
H^{(j)} = \bigoplus_{p-q=j} H^{p,q}(F).
\]
By the motivic axiom \cite{Kontsevich-Manin:GW}, 
$H^{(j)}$ is preserved by $(c_1(F)\star_0)$ and $H^{(j)} \star_0 H^{(k)} 
\subset H^{(j+k)}$. We consider Property $\O$ on $H^{(0)} = 
\bigoplus_{p=0}^{\dim F} H^{p,p}(F)$, $H^{\rm even}(F) = \bigoplus_{j\in 2\Z} 
H^{(j)}$ and $H^{\rm full}(F)$. 
\begin{theorem}
\label{thm:full_PropertyO} 
For every Fano manifold, 
we have $T_{H^{(0)}} = T_{H^{\rm full}(F)} = T_{H^{\rm even}(F)} = T$. 
Moreover, the following are equivalent: 
\begin{itemize} 
\item[(1)] $F$ has Property $\O$ on $H^{(0)}$. 
\item[(2)] $F$ has Property $\O$ on $H^{\rm even}(F)$. 
\item[(3)] $F$ has Property $\O$ on $H^{\rm full}(F)$. 
\end{itemize} 
\end{theorem} 

This theorem follows from the following two lemmas. 

\begin{lemma} 
For all $j$, 
the spectrum of $(c_1(F)\star_0)$ on $H^{(j)}$ is a subset of the 
spectrum of $(c_1(F)\star_0)$ on $H^{(0)}$. 
\end{lemma} 
\begin{proof}
The reason is the same as in Proposition \ref{prop:ambient_evaluation}; 
it suffices to note that $H^{(j)}$ is a module over the algebra $(H^{(0)},\star_0)$ 
and that $c_1(F)\in H^{(0)}$.  
\end{proof} 

\begin{lemma} 
Let $\lambda$ be an eigenvalue of $(c_1(F)\star_0)$ on $H^{(0)}$ of 
multiplicity one. Then $\lambda$ is not an eigenvalue of $(c_1(F)\star_0)$ 
on $H^{(j)}$ with $j\neq 0$.  
\end{lemma} 

\begin{proof} 
The proof here is analogous to the argument of 
Hertling--Manin--Teleman \cite{HMT}. 
Let $\psi_0\in H^{(0)}$ be a non-zero eigenvector 
of $(c_1(F)\star_0)$ of eigenvalue $\lambda$. 
By assumption, such $\psi_0$ is unique up to constant. 
Since we have $H^{(0)} = \C\psi_0 \oplus \im((c_1(F)\star_0) - \lambda)$, 
we can write
$1 = c \psi_0 + ((c_1(F)\star_0)-\lambda) \gamma$ 
for some $\gamma\in H^{(0)}$. 
Applying $(\psi_0\star_0)$ and noting that 
$\psi_0\star(c_1(F)\star_0 \gamma - \lambda \gamma) 
= \lambda \psi_0\star_0\gamma - \lambda \psi_0\star_0\gamma=0$, 
we obtain 
\[
\psi_0 = c\psi_0 \star_0 \psi_0. 
\] 
Thus $c\neq 0$. By replacing $\psi_0$ with $c\psi_0$, 
we may assume that $c = 1$, i.e.~$\psi_0 \star_0 \psi_0 = \psi_0$.

Suppose on the contrary that there exists a non-zero $\theta \in H^{(j)}$ 
with $j \neq 0$ such that
$c_1(F) \star_0 \theta = \lambda \theta$. 
By Poincar\'{e}--Serre duality, we can find $\sigma\in H^{(-j)}$ 
such that $(\theta, \sigma) \neq 0$. 
We now consider $\beta:= \theta \star_0 \sigma \in H^{(0)}$. 
Since $(\beta, 1) = (\theta, \sigma) \neq 0$, we have $\beta\neq 0$. 
We also have $c_1(F)\star_0 \beta = c_1(F) \star_0 \theta \star_0 
\sigma = T\theta \star_0 \sigma = T\beta$.  
Therefore $\beta$ is a non-zero eigenvector of $(c_1(F)\star_0)$. 
Our assumption implies that $\beta = e\psi_0$ for some $e \neq 0$. 
On the other hand, $\beta = \theta \star_0 \psi_0$ is nilpotent 
in quantum cohomology 
because $\theta$ is nilpotent (recall that $H^{(j)} \star_0 H^{(k)} 
\subset H^{(j+k)}$). 
This contradicts
the fact that $\psi_0$ is a non-zero idempotent. 
\end{proof} 

\begin{remark} 
Property $\O$ on the full cohomology group is useful when discussing 
Property $\O$ for the product. From the quantum K\"{u}nneth isomorphism 
and Theorem \ref{thm:full_PropertyO}, 
one can easily deduce that if Fano manifolds $F_1,F_2$ satisfy Property $\O$ 
(say, on $H^{\rm even}$), 
then $F_1\times F_2$ also satisfies Property $\O$. 
Note that $H^{\rm odd}(F_1) 
\otimes H^{\rm odd}(F_2)$ contributes to $H^{\rm even}(F_1 \times F_2)$ 
under the K\"unneth isomorphism, and we need Property $\O$ on the 
full cohomology group. 
\end{remark} 

Finally we discuss Gamma Conjecture I on the full cohomology group. 
The small quantum connection \eqref{eq:qconn} 
on the full cohomology group is still a flat connection in the usual sense, but 
it splits into the direct sum of flat connections on $H^{\rm even}(F)$ 
and $H^{\rm odd}(F)$; moreover it decomposes into flat connections on 
$H^{(j)}$, $j\in \Z$. If a Fano manifold $F$ satisfies Property $\O$, 
the spectrum of $(c_1(F)\star_0)$ on $\bigoplus_{j\neq 0} H^{(j)}$ 
is contained in $\{u : \Re(u) <T\}$ 
by the above lemmas. 
The general property of irregular connections of Poincar\'e rank one 
(see \cite[Propositions 3.2.1, 3.3.1]{GGI:gammagrass})  
shows that the subspace $\AA$ of flat sections and the principal asymptotic 
class $A_F$ appearing in \S\ref{subsec:principal_asymptotic}
remain the same if we consider the small quantum connection on $H^{(0)}$
or on $H^{\rm full}(F)$. This implies that $A_F$ lies in $H^{(0)} 
= \bigoplus_{p=0}^{\dim F} H^{p,p}(F)$, and that \emph{Gamma Conjecture I 
on the subspaces $H^{(0)}$, $H^{\rm even}(F)$, $H^{\rm full}(F)$ 
are all equivalent to each other}, provided that $F$ satisfies Property $\O$. 

\begin{remark} 
The discussion in this appendix suggests that the even part of quantum 
cohomology knows something about the odd part. 
Gamma conjecture also suggests a similar relation.  
If $F$ satisfies Gamma conjecture $I$, the even part of quantum 
cohomology knows the Gamma class $\Gg_F$; 
the Gamma class then gives the Chern classes $\ch_k(TF)$ 
(if the Euler constant 
$\gamma$ and $\zeta(k), k\ge 2$ are linearly independent over $\Q$) 
and in particular, the Euler number and $\dim H^{\rm odd}(F)$. 
\end{remark} 

\bibliographystyle{plain} 
\providecommand{\arxiv}[1]{\href{http://arxiv.org/abs/#1}{arXiv:#1}}

\end{document}